\documentclass[11pt]{amsart}
\usepackage[margin=1in]{geometry}
\usepackage{amsmath,amsfonts,amssymb,amsthm,epsfig,epstopdf,url,array,mathrsfs,dsfont,enumerate, fancyhdr, graphicx, proof, comment, multicol,tabularx,float,tikz,xypic,hyperref,indentfirst,appendix}
\hypersetup{
	colorlinks=true,	
  linkcolor=color1,	
    urlcolor=black,
	linktoc=all,
	citecolor=color1
}
\let\C\relax
\usepackage[color,matrix,arrow]{xy}
\newcommand{\ol}{\overline}
\newcommand{\ul}{\underline}
\newcommand{\C}{\mathbb{C}}
\newcommand{\Z}{\mathbb{Z}}
\newcommand{\Q}{\mathbb{Q}}

\newcommand{\Pbb}{\mathbb{P}}

\newcommand{\Ocal}{\mathcal{O}}
\newcommand{\Bcal}{\mathcal{B}}
\newcommand{\Dcal}{\mathcal{D}}

\newcommand{\Hrm}{\mathrm{H}}

\newcommand{\Lrm}{\mathrm{L}}
\newcommand{\Hom}{\mathrm{Hom}}

\newcommand{\im}{\mathrm{Im}}
\newcommand{\gr}{\mathrm{gr}}

\newcommand{\rank}{\mathrm{rank}}

\newcommand{\lieb}{\mathfrak{b}}
\newcommand{\lien}{\mathfrak{n}}
\newcommand{\hfrak}{\mathfrak{h}}
\newcommand{\gfrak}{\mathfrak{g}}

\newcommand{\afrak}{\mathfrak{a}}

\newcommand{\Rcal}{\mathcal{R}}
\newcommand{\Nscr}{\mathscr{N}}
\newcommand{\ebf}{\mathbf{e}}
\newcommand{\hbf}{\mathbf{h}}
\newcommand{\fbf}{\mathbf{f}}

\newcommand{\Scal}{\mathcal{S}}

\newcommand{\Hscr}{\mathscr{H}}

\newcommand{\SL}{\mathrm{SL}}

\newcommand{\wt}{\widetilde}
\newcommand{\e}{\mathrm{e}}
\newcommand{\wh}{\widehat}
\newcommand{\alg}{\mathrm{alg}}
\newcommand{\ind}{\mathrm{ind}}

\newcommand{\Ker}{\mathrm{Ker}}
\newcommand{\ord}{\mathrm{ord}}
\newcommand{\Hilb}{\mathrm{Hilb}}

\newcommand{\sign}{\mathrm{sign}}

\newcommand{\ch}{\mathrm{ch}}
\newcommand{\st}{\mathrm{st}}
\newcommand{\triv}{\mathrm{triv}}

\theoremstyle{plain}
\newtheorem{thm}{Theorem}[section]
\newtheorem{prop}[thm]{Proposition}
\newtheorem{coro}[thm]{Corollary}
\newtheorem{lem}[thm]{Lemma}
\newtheorem{conj}[thm]{Conjecture}
\theoremstyle{definition}
\newtheorem{defn}{Definition}[section]
\theoremstyle{remark}
\newtheorem{rem}{Remark}[section]
\newtheorem{ex}[rem]{Example}
\definecolor{color1}{HTML}{318CE7}
\definecolor{color2}{HTML}{F88379}

\title{Rational Cherednik algebras and knot invariants: filtrations}
\author{Xinchun Ma} 
\email{xinchma@gmail.com}
\begin{document}
\maketitle
\begin{abstract}
In \cite{gors}, the HOMFLY polynomial of the $(m,n)$ torus knot $T_{m,n}$ is extracted from the hook isotypic components within the graded character of the finite-dimensional Cherednik algebra representation $\Lrm_{\frac{m}{n}}$. It is furthermore conjectured in \textit{loc. cit.} that one can obtain the triply-graded Khovanov-Rozansky homology of $T_{m,n}$ by considering a certain filtration on $\Lrm_{\frac{m}{n}}$. In this paper, we show that two of the proposed candidates, the algebraic filtration and the inductive filtration, are equal. Our proof is combinatorial and elementary, transforming the filtration identity into a linear algebra problem on coinvariant algebras.
\end{abstract}
\section{Introduction}
The HOMFLY polynomial is a link invariant defined by skein relations in the form of a two-variable polynomial, generalizing the one-variable Alexander polynomial and the Jones polynomial. In the early 2000s, these link polynomials were realized as Euler characteristics of various link homology theories. This process was inspired by the ideas of ``categorification"; for a survey, see \cite{khovanovsurvey}. The HOMFLY polynomial was categorified by the Khovanov–Rozansky (KhR) link homology using matrix factorization \cite{khr}. 

Soon after, the KhR homology was reconstructed \cite{khovanovsoergel}: for any element $\gamma$ in the braid group, one can associate to it a Rouquier complex built on Soergel bimodules labeled by simple twists. Take the Hochschild homology of each term of the Rouquier complex to form a new complex. The cohomology of this new complex, usually denoted by $\Hrm\Hrm\Hrm(\ol{\gamma})$, only depends on the link closure $\ol{\gamma}$ of the braid $\gamma$ and recovers the link homology defined in \cite{khr}. Furthermore, under this definition, the KhR homology is naturally {triply graded} by the  cohomological grading, the Hochschild homological grading and the internal grading of the Soergel bimodules. Interestingly, this triply graded structure appears in various seemingly unrelated algebro-geometric and representation-theoretic contexts; for a survey, see \cite{gorskysummerschool}.

This paper will focus on a representation-theoretic model of the KhR homology of torus knots. Let $m,n$ be a coprime positive integer pair and let $\hfrak$ denote the $(n-1)$-dimensional standard representation of the symmetric group $S_n$. Introduced by Etingof and Ginzburg \cite{etingof-ginzburg}, the rational Cherednik algebra  $\Hrm_{\frac{m}{n}}$ (also called rational DAHA), is a deformation of $D(\hfrak)\rtimes S_n$ at the parameter ${\frac{m}{n}}$. Here $D(\hfrak)$ is the ring of differential operators on $\hfrak$. The algebra $\Hrm_{\frac{m}{n}}$ has a unique finite-dimensional irreducible representation, which we denote by $\Lrm_{\frac{m}{n}}$. Under the action of the Euler field (cf. (\ref{h})), $\Lrm_{\frac{m}{n}}$ decomposes into a direct sum of eigenspaces $\Lrm_{\frac{m}{n}}=\bigoplus_i \Lrm_{\frac{m}{n}}(i)$, allowing us to define the graded dimension $\ch_q(\Lrm_{\frac{m}{n}})=\sum_i \dim(\Lrm_{\frac{m}{n}}(i))q^i$. In \cite{gors}, the authors show that the HOMFLY polynomial of the $(m,n)$ torus knot $T_{m,n}$ and the graded character of the hook $S_n$-isotypic components of $\Lrm_{\frac{m}{n}}$ coincide:
\begin{thm}\label{gorsthm}
    \cite[Theorem 1.1]{gors}  
    \begin{align}\label{homfly}
   \mathrm{HOMFLY}
_{a,q}(T_{m,n})=a^{(n-1)(m-1)}\sum_{i=0}^{n-1}a^{2i}\ch_q\big(\Hom_{S_n}(\wedge^i(\hfrak),\Lrm_{\frac{m}{n}})\big).
 \end{align}
In this equality, the following gradings correspond:
		\begin{align*}
  		\text{Hochschild homological $a$-grading} &\leftrightarrow \text{degree of $\wedge^\bullet \hfrak$}\\
			\text{ internal $q$-grading}&\leftrightarrow  \text{Euler grading}
		\end{align*}
\end{thm}

The full triply graded dimension of the KhR homology is known as the ``link superpolynomial" \cite{Aganagic-Shakirov,Cherednik_superpoly,gn}. 
The main conjecture of \cite{gors} is that the third grading on the KhR homology corresponds to a filtration on $\Lrm_{\frac{m}{n}}$:
\begin{conj}\label{triply}
    \cite[Conjecture 1.2]{gors} There exists a filtration $F$ on $\Lrm_\frac{m}{n}$ whose associated grading corresponds to the homological $t$-grading on $\Hrm\Hrm\Hrm(T_{m,n})$ and yields the refined identity:
   \[\ch_{a,q,t} (\Hrm\Hrm\Hrm(T_{m,n}))=a^{(n-1)(m-1)}\sum_{i=0}^{n-1}a^{2i}\ch_{q,t}\big(\Hom_{S_n}(\wedge^i(\hfrak),\gr^F\Lrm_\frac{m}{n})\big).\]
\end{conj}
In \textit{op. cit.}, Three potential candidates for this filtration are proposed, and it is conjectured that they coincide:
\begin{itemize}
    \item The algebraic filtration (cf. Section \ref{subsection_F_alg}) $F^\alg$  is defined by reshuffling weight spaces under the Euler field and orthogonal complements of powers of the ideal $\afrak$ generated by nonconstant homogeneous symmetric polynomials under the Dunkl form.
    \item The inductive filtration $F^\ind$ (cf. Section 
    \ref{subsection_F_ind}) is defined inductively using the shift functor from $\Hrm_{\frac{m}{n}}$-modules to $\Hrm_{\frac{m}{n}+1}$-modules and the ``flipping" isomorphism $(\Lrm_{\frac{m}{n}})^{S_n}\cong (\Lrm_{\frac{n}{m}})^{S_m}$. It is defined on $\Lrm_{\frac{m}{n}}$ only when $m$ is larger than and coprime to $n$ but on $(\Lrm_{\frac{m}{n}})^{S_n}$ for all positive $m$ coprime to $n$.
    \item The geometric filtration $F^{\text{geom}}$ is defined in terms of the perverse filtration and the cohomological grading on the cohomology of a Hitchin fiber isomorphic to the compactified Jacobian of the planar singular curve $y^m=x^n$. 
    \end{itemize}
    The identity $F^\alg=F^{\text{geom}}$ has been proved in \cite{oy}. The goal of this paper is to relate $F^\alg$ and $F^\ind$.
\begin{conj}\cite[Conjecture 4.12]{gors}\label{GORSconj}
 Suppose $m$ is a positive integer coprime to $n$. We have that \[\begin{cases}
 F_i^\alg\Lrm_{\frac{m}{n}}=F_i^\ind\Lrm_{\frac{m}{n}} &\forall i\in \Z\text{ when }m>n\\
 F_i^\alg(\Lrm_\frac{m}{n})^{S_n}=F_i^\ind (\Lrm_\frac{m}{n})^{S_n} &\forall i\in \Z
 \end{cases}\]
\end{conj}
\begin{thm}\label{main}(Theorem \ref{thm})
Conjecture \ref{GORSconj} holds.
\end{thm}
\begin{rem}
    Conjecture \ref{triply} has been proved with respect to $F^\ind$ (and hence $F^\alg$) in \cite{ma2} when this article is under revision.
\end{rem}
\subsection{Proof strategy}
The general method we use to prove Theorem \ref{main} is by simultaneous inductions on $c:={\frac{m}{n}}$ and on the filtration index $i$ in $F^\alg_i$ and  $F^\ind_i$. 

We start by exploring a subspace of $\C[\hfrak]$ orthogonal to the ideal $\afrak$ used in the definition of $F^\alg$ under the Dunkl form, consisting of ``$c$-harmonic polynomials" (Section \ref{harmonic}). One recurrent theme throughout the article is that any elements in $\Lrm_c$ is a sum of products of symmetric polynomials and $c$-shifted harmonic polynomials. 

One of our observations is that $F^\alg$ and $F^\ind$ are Kazhdan filtrations associated to the Euler grading and two other filtrations respectively: the power filtration $F^\afrak$ and another inductive filtration $F^{\ul{\ind}}$ (Section \ref{sectionfiltration}). 

In Sections \ref{sectionfiltration} and \ref{section_first_relations}, we lay out some useful properties of $F^\afrak$ and $F^{\ul{\ind}}$ and their first relations, including their compatibility with a $\mathfrak{sl}_2$-triple, the inclusion $F^{\ul{\ind}}\subset F^\afrak$, $F_0^{\ul{\ind}}= F_0^\afrak$, etc.

The representation $\Lrm_c$ is a quotient of $\C[\hfrak]$ modulo an explicit ideal $I_c$. In Section \ref{subsection_orthogonal_complement}, we study the identity $F^\afrak=F^{\ul{\ind}}$ using induction on $c$ over elements orthogonal to $(I_{c-1}^{S_n})$, the ideal generated by $I_{c-1}^{S_n}$. 
It is much harder to show $F^\afrak=F^{\ul{\ind}}$ over $(I_{c-1}^{S_n})$ (cf. (\ref{S}) for a precise description). In Proposition \ref{yphi} we transform this question to studying whether the Dunkl operators reduce the filtration index $i$ in $F^\afrak_i$ for elements in $(I_{c-1}^{S_n})$.

In Section \ref{sectionkernel}, we first inspect $I_c$ using the residue description provided in \cite{dunklintertwining, chmutovaetingof, gorskyarc}. Using symmetric polynomials as basis and $c$-shifted harmonic polynomials as coefficients, we are able to translate the question at the end of the last paragraph for $i=1$ into a linear algebra problem. Explicitly, we introduce two matrices $A$ and $B$ that resemble each other. We conduct dimension counts of certain subspaces in direct sums of copies of coinvariant algebras to show that $\im(A)=\Ker(B)$ (Section \ref{kernel}, \ref{transform}). The dimension computation is derived from the appearance of certain distributive lattices (Section \ref{linearalgebra}, Appendix \ref{appendix}). The special case of $m=2n-1$ can also be studied via the cohomology of the Springer fiber at a minimal nilpotent element (Section \ref{minimalspringer}). In Section \ref{section_F_i}, the identity $F_i^{\ul{\ind}}= F_i^\afrak$ for $i>1$ is deduced in a similar way  using ingredients developed in Section \ref{sectionkernel} with more intricate treatment. 


\subsection{Remarks}Our proof is highly combinatorial and it would be interesting to geometrize it. It would also be more satisfying to find a conceptual explanation for the occurrence of our distributive lattices.

Throughout the article, we will use the symbol ``$\spadesuit$" to signal conventions and assumptions that will be used repeatedly. 

\subsection*{Acknowledgements} The author would like to thank Victor Ginzburg, Eugene Gorsky, Thomas Hameister, Do Kien Hoang, Yixuan Li, Linus Setiabrata, Minh-Tam Trinh and Tom Wilde for interesting discussions and anonymous referees for helpful feedback. This work was partially supported by the National Science Foundation under Award No. DMS-2503504. 
\section{Rational Cherednik algebra and $c$-shifted harmonic polynomials}\label{harmonic}
\subsection{Rational Cherednik algebras}
Fix a positive integer $n$. We work over $\C$ throughout, and set up the following notations: \\
    $W=S_n$: the symmetric group on a set of size $n$;\\
    $\hfrak=\C^{n-1}$: the standard representation of $W$;\\
    $S\subset W$: the set of reflections.
    \begin{defn}
    	For any $c\in \C$, the rational Cherednik algebra $\mathrm{H}_c:=\Hrm_c(W,\hfrak)$ is the $\C$-algebra generated by $\hfrak$, $\hfrak^*$ and $W$ with relations
	\begin{align*}
		&[x,x']=[y,y']=0, \quad 
		wxw^{-1}=w(x), \quad wyw^{-1}=w(y)\\
		&[y,x]=x(y)-\sum_{s\in S}c\langle \alpha_s,y\rangle\langle \alpha_s^\vee, x\rangle s
	\end{align*}
	where $x,x'\in \hfrak^*$, $y,y'\in \hfrak$, $w\in W$. Here $\alpha_s$, resp. $\alpha_s^\vee$, is the root, resp. coroot, associated to $s$ and $(-,-)$ is the pairing between $\hfrak$ and $\hfrak^*$.
        \end{defn}

\subsection{The Dunkl embedding}Let $\ol{\hfrak}=\triv\oplus \hfrak$ and  $\ol{\hfrak}_r=\ol{\hfrak}\setminus \cup_{s\in S}\{\alpha_s=0\}$ be the regular part of $\ol{\hfrak}$. Let $\Dcal(\ol{\hfrak}_r)$ be the ring of differential operators on ${\hfrak}_{r}$. We take $x_1,\cdots,x_n$ to be the standard basis of $\ol{\hfrak}^*$. Define the Dunkl operators \cite{Dunkl_1989}:\[y^{(c)}_i:=\frac{\partial}{\partial {x_i}}-c\sum_{j\neq i}\frac{1-(ij)}{x_i-x_j}.\]
Below we may drop the superscript $^{(c)}$ when there is no ambiguity.  The algebra ${\Hrm}_c$ can be identified with the subalgebra of $\Dcal(\ol{\hfrak}_{r})\rtimes W$ generated by $x_i-x_{i+1}$, $y_i-y_{i+1}$ for $1\leq i\leq n-1$ and $W$.
\subsection{Finite-dimensional representations}
Take the polynomial representation of $\Hrm_c$: \[\C[\hfrak]\cong \Hrm_c \otimes_{S(\hfrak)\rtimes W}\C,\]
which has a maximal proper submodule with associated quotient denoted by $\Lrm_c$.
\begin{ex}
	 When $c=0$, $\Hrm_0=\Dcal(\hfrak)\rtimes W$ whose polynomial representation $\C[\hfrak]$ is irreducible.
	\end{ex}
\begin{thm}\label{beg}
(\cite[Theorem 1.2]{beg})	When $c=\frac{m}{n}$ for positive integer $m$ coprime to $n$, the only irreducible finite-dimensional representation of $\Hrm_c$ is $ \Lrm_c$. Moreover, only when $c=\frac{m}{n}$ for integer $m$ coprime to $n$ does $\Hrm_c$ have finite-dimensional representations.
	\end{thm}
 \subsection{Dunkl form}
On ${\Hrm}_c$, there is the Fourier transform, which is an algebra automorphism: 
\begin{align}
    \label{fourier}
\Phi_c(x_i-x_j)=y_i-y_j, \quad \Phi_c(y_i-y_j)=x_j-x_i, \quad \Phi_c(w)=w.\end{align}
\begin{defn}
The Dunkl bilinear form on $\C[\hfrak]$ is defined by
\[(-,-)_c: \C[\hfrak]\times\C[\hfrak]\to\C,\quad (f,g)_c=[\Phi_c(f)g]|_{x_i=0,\forall i}.\]
\end{defn}
From the definition, we see that for any $\phi,\psi\in \C[\hfrak]$,
\begin{align}\label{xy}
    ((x_i-x_j)\phi,\psi)_c=(\phi,(y_i-y_j)\psi)_c.
    \end{align}
    On the other hand, since $w x_i w^{-1}=x_{w(i)}$ and $w y_i w^{-1}=y_{w(i)}$ for any $w\in W$, we also have
    \begin{align}
        \label{invariant}
        (w(\phi),w(\psi))_c=(\phi,\psi)_c.
    \end{align}
     For generic $c$, the Dunkl form is non-degenerate. In general, let $I_c$ be the kernel of $(-,-)_c$. The resulting quotient $\C[\hfrak]/I_c$ is exactly isomorphic to $\Lrm_c$ (\cite[Proposition 2.34]{dunklopdam}). 
    \begin{ex}\label{ex_L_k/2}
    	Let $c=\frac{k}{2}$ and $s=(12)$. In this case, $y_1-y_2=\frac{\partial}{\partial_{x_1}}-\frac{\partial}{\partial_{x_2}}-2c\frac{1-s}{x_1-x_2}$ and  \[(y_1-y_2)\big((x_1-x_2)^\ell\big)=\begin{cases}
    	    (2\ell-4c) (x_1-x_2)^{\ell-1}\quad &\text{when $\ell$ is odd}\\
            2\ell(x_1-x_2)^{\ell-1} &\text{when $\ell$ is even}
    	\end{cases}.\] As a result, \[\Lrm_{\frac{k}{2}}=\begin{cases}\C[\hfrak]/((x_1-x_2)^k)&\text{ when $k$ is odd}\\
        \C[\hfrak]&\text{ when $k$ is even}
        \end{cases}\]
    \end{ex}
    Let $\hat{W}$ denote the set of irreducible $W$-representations. For every $\sigma\in \hat{W}$, define the associated idempotent $e_\sigma=\frac{\chi_\sigma(1)}{n!}\sum_{w\in W} \chi_\sigma(w)^{-1}w$. Under the $W$-action, $\Lrm_c$ decomposes into isotypic components:
\[\displaystyle \Lrm_c=\bigoplus_{\sigma\in \hat{W}} \e_\sigma \Lrm_c,\quad \e_\sigma \Lrm_c=\sigma\otimes \Hom_W(\sigma,\Lrm_c).\]
 We set $\e:=\e_\triv$ and $\e_-:
 =\e_\sign$. Then  $\e\Lrm_c=(\Lrm_c)^{S_n}$ and $\e_-\Lrm_c=(\Lrm_c)^{\sign}$. 
 
 For any two subspaces $U$, $V$ of $\C[\hfrak]$ or $\Lrm_c$, we write $U\perp_c V$ when $U$ and $V$ are orthogonal to each other with respect to $(-,-)_c$. 
\begin{lem}
\label{isotypic}For $\sigma\ncong \sigma'$, $\e_\sigma\Lrm_c\perp_c \e_{\sigma'}\Lrm_c$.
\end{lem}
\begin{proof}
It is a consequence of the property (\ref{invariant}).
\end{proof}
\textbf{$\spadesuit$ From now on we will assume that $c=\frac{m}{n}$ for positive integer $m$ coprime to $n$. } For the rest of Section \ref{harmonic}, we will further assume that $c>1$.
\subsection{$c$-Shifted harmonic polynomials}
Inside $\C[\hfrak]$, take the ideal generated by $\C[\hfrak]^W_+$: \[\afrak:=(\C[\hfrak]^W_+).\] Let $\delta:=\prod_{i<j}(x_i-x_j)$ denote the Vandermonde determinant.
Define the space of $c$-shifted Harmonic polynomials to be \[\Hscr_c:=  \C[y_1,\dots, y_n]\delta.\] 
When $c=0$, this space consists of $W$-harmonic polynomials \cite[6.3]{chrissginzburg}. The main result of this section is that $\Hscr_c=\afrak^{\perp_c}$ when $c>1$, parallel to the well-known equality $\Hscr_0=\afrak^{\perp_0}$.
\begin{lem}\label{higherpower}
    For non-negative integers $b_2,\cdots, b_n$, if $\sum_{i=2}^n b_i=\frac{n(n-1)}{2}$ and $\{b_2,\cdots, b_n\}\neq  \{1,2,\cdots, n-1\}$, then $y_2^{b_2}y_3^{b_3}\cdots y_n^{b_n}\delta=0$.
\end{lem}
\begin{proof}
Take non-negative integers $b_2,\cdots, b_n$ satisfying assumptions of the lemma. For degree reasons $y_2^{b_2}y_3^{b_3}\cdots y_n^{b_n}\delta$ is a constant. Therefore for any $\sigma\in W$,
\[y_2^{b_2}y_3^{b_3}\cdots y_n^{b_n}\delta=\sigma(y_2^{b_2}y_3^{b_3}\cdots y_n^{b_n}\delta)=y_{\sigma(2)}^{b_2}y_{\sigma(3)}^{b_3}\cdots y_{\sigma(n)}^{b_n} ((-1)^{\sign(\sigma)}\delta).\]
From this we deduce that
\begin{align}
    \label{sigma}
\sum_{\sigma\in W} (-1)^{\sign(\sigma)}\sigma(y_2^{b_2}y_3^{b_3}\cdots y_n^{b_n})\delta=n! y_2^{b_2}y_3^{b_3}\cdots y_n^{b_n}\delta.
\end{align}

On the other hand, there is \begin{align}
    \label{skew}
\C[\hfrak]^{\sign}=\delta\C[\hfrak], 
\end{align}
and thus the degree $\frac{n(n-1)}{2}$ skew-symmetric polynomial $\sum_\sigma (-1)^{\sign(\sigma)}\sigma(x_2^{b_2}x_3^{b_3}\cdots x_n^{b_n})$ has to equal $\lambda \delta$ for some constant $\lambda$. However, we know that $\delta$ expands into
\begin{align}
    \label{deltaexpand}
\sum_{\sigma\in W} (-1)^{\sign(\sigma)}x_{\sigma(2)}x_{\sigma(3)}^2\cdots x_{\sigma(n)}^{n-1}.\end{align}
Since $\{b_2,\cdots, b_n\}\neq  \{1,2,\cdots, n-1\}$, by comparing the degrees of terms in $\delta$ and \[\sum_\sigma (-1)^{\sign(\sigma)}\sigma(x_2^{b_2}x_3^{b_3}\cdots x_n^{b_n}),\] we conclude that $\lambda=0$ and the lemma follows.
\end{proof}
\begin{prop}\label{perp}
When $c>1$, we have that $\Hscr_c=\afrak^{\perp_c}$.
\end{prop}
\begin{proof}
The image of  $\delta$ under the action of any symmetric polynomial in $y_i$'s is still skew-symmetric and hence has to be $0$ given (\ref{skew}). Therefore, $\Hscr_c\perp_c \afrak$. Moreover, we recall that $\C[\hfrak]/\afrak\cong \C[W]$. Thus, to prove the proposition, it suffices to show that $\dim \Hscr_c=n!$ and $\Hscr_c\cap I_c=\{0\}$.

\underline{Claim}: $y_2y_3^2\cdots y_n^{n-1}\delta\neq 0$. 

When $c>1$, by \cite[4.3]{beg}, one has that $\delta \e\Lrm_c\cong \e \Lrm_{c-1}$ with $\delta$ mapping to $1$. Therefore $\delta$ is nonzero in $\e\Lrm_c$. Since the Dunkl form is nondegenerate on $\Lrm_c$, there exists a polynomial $\phi$ such that $(\delta,\phi)_c\neq 0$. By Lemma \ref{isotypic}, we may assume $\phi\in \e_-\Lrm_c=\delta \e\Lrm_c$. For degree reasons, we must have that $\phi$ is a nonzero multiple of $\delta$. As a result, we have that $(\delta,\delta)_c\neq 0$, which by (\ref{deltaexpand}) expands to 
\begin{align}
    \label{deltapair}
(\delta,\delta)_c=\left(\sum_{\sigma\in W}(-1)^{\sign(\sigma)} y_{\sigma(2)}y_{\sigma(3)}^2\cdots y_{\sigma(n)}^{n-1}\right)\delta.
\end{align}
By (\ref{sigma}), the nonzero expression (\ref{deltapair}) equals
$n!y_2y_3^2\cdots y_n^{n-1}\delta$
and the claim follows.

We continue to show that, for $0\leq a_i\leq i-1$, the nonzero elements $y_2^{a_2}y_3^{a_3}\dots y_{n}^{a_{n}} \delta$ are linearly independent. 

Suppose otherwise, i.e. 
\[\displaystyle \sum_{\sum a_i=N} \alpha_{a_2,\dots, a_{n}} y_2^{a_2}y_3^{a_3}\dots y_{n}^{a_{n}} \delta=0,\]
where $0\leq a_i\leq i-1$ and $N\in \Z_{\geq 0}$. Let $(A_2,\dots, A_n)$ be the maximal element inside the set 
\[\{(a_1,\dots, a_{n-1})|\sum a_i=N,\  0\leq a_i\leq i-1,\alpha_{a_2,\dots, a_{n}}\neq 0\}\]
under the lexicographical order\footnote{i.e. for $(a_1,\dots, a_{n-1})$ and $(a'_1,\dots, a'_{n-1})$, if for some $1\leq i\leq n-1$, $a_1=a'_1,\cdots,a_i=a'_i$ and $a_{i+1}>a'_{i+1}$ then $(a_1,\dots, a_{n-1})>(a'_1,\dots, a'_{n-1})$}.

Then \[y_2^{1-A_2}y_3^{2-A_3}\cdots y_n^{n-1-A_n} \displaystyle \sum_{\sum a_i=N} \alpha_{a_2,\dots, a_{n}} y_2^{a_2}y_3^{a_3}\dots y_{n}^{a_{n}} \delta=0.\]

As a consequence of Lemma \ref{higherpower}, we have that 
\[0=y_2^{1-A_2}y_3^{2-A_3}\cdots y_n^{n-1-A_n} \displaystyle \sum_{\sum a_i=N} \alpha_{a_2,\dots, a_{n}} y_2^{a_2}y_3^{a_3}\dots y_{n}^{a_{n}} \delta=\alpha_{A_2,\dots, A_n}y_2y_3^2\cdots y_n^{n-1}\delta\] 
This contradicts the claim above. Therefore, the $n!$ elements $y_2^{a_2}y_3^{a_3}\dots y_{n}^{a_{n}} \delta$ for $0\leq a_i\leq i-1$ are linearly independent and span the space $\Hscr_c$. Therefore $\dim(\Hscr_c)=n!$.

We proceed to show that $I_c\cap \Hscr_c=\{0\}$.
As discussed above, \[(y_2^{1-a_2}y_3^{2-a_3}\dots y_{n}^{n-1-a_{n}})(y_2^{a_2}y_3^{a_3}\dots y_{n}^{a_{n}}) \delta\]
is a nonzero constant. Moreover, if $\phi\in I_c$, then $(y_i-y_{i+1})\phi\in I_c$ for any $1\leq i\leq n-1$ because of (\ref{xy}). Since nonzero constants are not contained in $I_c$, we conclude that $y_2^{a_2}y_3^{a_3}\dots y_{n}^{a_{n}} \delta\notin I_c$ for all $0\leq a_i\leq i-1$. This concludes the proof of the proposition.
\end{proof}
\begin{rem}
 In view of the proof, we actually have that the identity $\Hscr_c=\afrak^{\perp_c}$ holds whenever $\delta\notin I_c$.
Also, as pointed out by Ginzburg, statement like Proposition \ref{perp} likely holds for a general Coxeter system.
\end{rem}
The following simple lemma will play an important role later. 
\begin{lem}\label{basis}
    The polynomials $x_2^{a_2}x_3^{a_3}\dots x_{n}^{a_{n}}$, for  $0\leq a_i\leq i-1$, form a basis of $\C[\hfrak]/\afrak$.  Moreover, any $\phi\in \C[\hfrak]$ can be uniquely expressed as $\phi=\sum h_i \psi_i$ with $h_i\in \Hscr_c$ and $\psi_i\in \C[\hfrak]^W$. In other words, $\C[\hfrak]=\Hscr_c\cdot \C[\hfrak]^W$.
\end{lem}
\begin{proof}
It follows from Proposition \ref{perp} and the fact that $\C[\hfrak]$ is the Galois extension of $\C[\hfrak]^W$ with respect to the polynomial $\prod_{i=1}^n (x-x_i)$ with basis given by $x_2^{a_2}x_3^{a_3}\dots x_{n}^{a_{n}}$, $0\leq a_i\leq i-1$.
\end{proof}
 $\spadesuit$ To avoid extra notations, below we  will simply use $\afrak$ and $\Hscr_c$ to denote \textbf{their images} under $\C[\hfrak]\to \Lrm_c$ when there is no ambiguity. Under this convention, we have a direct sum decomposition \[\Lrm_c=\afrak\oplus^{\perp_c} \Hscr_c.\]

The next corollary of Lemma \ref{basis} will be used in the proof of Proposition \ref{yphi}.

Under the $W$-action on $\Hscr_c$ permuting the Dunkl operators, $\Hscr_c$ is isomorphic to the regular representation of $W$. Decompose $\Hscr_c=\bigoplus_{\sigma\in\hat{W}} \Hscr_c^\sigma$ so that $\Hscr_c^\sigma=\sigma\otimes \Hom_W(\sigma,\Hscr_c)$. 
\begin{coro}\label{coro_isotypic}
	For all $\sigma\in\hat{W}$, $\e_\sigma\Lrm_c=\Hscr_c^\sigma \cdot \e\Lrm_c$.
\end{coro}
\begin{proof}
	Clearly $\Hscr_c^\sigma \cdot \e\Lrm_c\subset \e_\sigma\Lrm_c $.  Moreover, it follows from Lemma \ref{basis} that $\Lrm_c=\Hscr_c\cdot \e\Lrm_c=(\bigoplus_\sigma \Hscr_c^\sigma)\cdot \e\Lrm_c=\bigoplus_\sigma \Hscr_c^\sigma \cdot \e\Lrm_c$ and hence the opposite inclusion follows.
\end{proof}
\if force \begin{proof}
    $\Phi_c(x_2^{1-a_2}x_3^{2-a_3}\dots x_{n}^{n-a_{n}})\Phi_c(x_2^{a_2}x_3^{a_3}\dots x_{n}^{a_{n}}) \delta\neq 0$  implies  $x_2^{a_2}x_3^{a_3}\dots x_{n}^{a_{n}}\notin \Hscr_c^{\perp_c}=\afrak$. On the other hand, the linearly independency of  $\Phi_c(x_2^{a_2}x_3^{a_3}\dots x_{n}^{a_{n}})\delta $, $0\leq a_i\leq i-1$ imply $x_2^{a_2}x_3^{a_3}\dots x_{n}^{a_{n}}$ mod $\afrak$ $0\leq a_i\leq i-1$ are linearly independent. Since the number of these polynomials matches with the dimension of $\Hscr_c$, the corollary follows.
\end{proof}\fi
\section{The power filtration and the inductive filtration}\label{sectionfiltration}
\subsection{An $\mathfrak{sl}_2$-triple}
Let $p_i:=\sum_{j=1}^n x_j^i$ be the power sum symmetric polynomial of degree $i$. Then $\C[\hfrak]^W= \C[p_2,\dots,p_n]$.

The following three elements in $\Hrm_c$ form an  $\mathfrak{sl}_2$-triple:
\[\mathbf{e}=\frac{1}{2n}\sum_{i<j} (x_i-x_j)^2,\ \mathbf{f}=-\Phi_c(\mathbf{e})=-\frac{1}{2n}\sum_{i<j} (y_i-y_j)^2,\]
\[\mathbf{h}=\frac{1}{2n}\sum_{i<j}\bigg((x_i-x_j)(y_i-y_j)+(y_i-y_j)(x_i-x_j)\bigg).\]
\[[\hbf,\ebf]=2\ebf,\quad [\hbf,\fbf]=-2\fbf,\quad [\ebf,\fbf]=\hbf.\]
The left action of $\mathbf{h}$ induces a decomposition of $\Lrm_c$ into weight spaces:
\[\Lrm_c=\bigoplus_k \Lrm_c(k), \quad \Lrm_c(k)=\{v\in \Lrm_c|\hbf v=kv\}.\]
\begin{lem}\label{h}
Let $\mathrm{d}:=\frac{(m-1)(n-1)}{2}$.
Suppose $\phi\in\C[\hfrak]$ is a homogeneous polynomial. In $\Lrm_c$, there is \[\hbf\cdot \phi=(\deg(\phi)-\mathrm{d})\phi.\]As a result, $\mathrm{min}(\{k\in\Z, \Lrm_c(k)\neq 0\})=-\mathrm{d}$.
\end{lem}
\begin{rem}
     The quantity $\mathrm{d}$ is the delta invariant of the singular curve $\{x^m=y^n\}$ at $0$.
\end{rem}
\begin{proof}
We compute that 
\begin{align}
    \mathbf{h}=&\frac{1}{2n}\sum_{i<j} \left((x_i-x_j)(y_i-y_j)+(y_i-y_j)(x_i-x_j)\right)\nonumber\\
    =&\frac{1}{2n}\sum_{i<j} \left(2(x_i-x_j)(y_i-y_j)+2-\sum_{s\in S}c\langle \alpha_s,x_i-x_j\rangle\langle  y_i-y_j,\alpha_s^\vee\rangle s \right)\nonumber\\
    =&\frac{1}{2n}\sum_{i<j}  2(x_i-x_j)(y_i-y_j)+\frac{1}{2n}\left(2|S|-2nc \sum_{s\in S} s\right) .\label{formula_h}
\end{align}
The last identity is because $\langle\alpha_s,x_i-x_j\rangle$ equals $2$ when $\alpha_s=x_i-x_j$ and equals $\pm 1$ when $s=(ik)$ or $(jk)$ for some $k\neq i,j$. 

On $1\in \Lrm_c$, $y_i-y_j$ acts by $0$ and $\sum_{s\in S} s$ acts by $|S|$. Therefore,
\[\hbf\cdot 1=\frac{1}{n}(1-nc)|S|=-\frac{(m-1)(n-1)}{2}=-\mathrm{d}.\]
Moreover, an almost identical computation as \cite[(2.6),(2,7)]{beg2} shows that:
\[\text{ in }\Hrm_c,\ [\hbf,x]=x,\ \forall x\in\hfrak^*.\]
Therefore, $[\hbf,\phi]=\deg(\phi)\phi$ in $\Hrm_c$, for all $\phi\in\C[\hbf]$. Altogether, in $\Lrm_c$ there is
\[\hbf\cdot \phi=[\hbf,\phi]\cdot 1+\phi (\hbf\cdot 1)=(\deg(\phi)-\mathrm{d})\phi.\qedhere\]
\end{proof}
\subsection{Fourier transform on $\Lrm_c$} The action of the $\mathfrak{sl}_2$ triple $\{\ebf,\fbf,\hbf\}$ on $\Hrm_c$ satisfies: 
\[[\ebf-\fbf,x_i-x_{i+1}]=[-\fbf,x_i-x_{i+1}]=y_i-y_{i+1}, [\ebf-\fbf,y_i-y_{i+1}]=-(x_i-x_{i+1}),\quad i=1,\dots,n-1.\]
Moreover, this action is locally finite and integrable. Note that the matrix $M=\begin{pmatrix}
    0&1\\-1&0
\end{pmatrix}$ satisfies $\e^{\frac{\pi}{2}M}=M$. Therefore, the Fourier transform (\ref{fourier}) can also be expressed by
    \begin{align}
        \label{phief}
    \Phi_c=\mathrm{Ad}(\e^{\frac{\pi}{2}(\mathbf{e}-\mathbf{f})}): \Hrm_c\to \Hrm_c.
    \end{align}
Since the $\mathfrak{sl}_2$ action on $\Lrm_c$ is locally finite, we can also define:
\begin{align}
\label{phionL}
\ol{\Phi}_c:\Lrm_c\to \Lrm_c,\quad \phi\mapsto \e^{\frac{\pi}{2}(\mathbf{e}-\mathbf{f})}\phi\end{align}
where $\e^{\frac{\pi}{2}(\mathbf{e}-\mathbf{f})}$ acts by left action. 

Such an observation has been studied in \cite[Remark after 3.8]{beg2} and \cite[Corollary 5.3]{etingof-ginzburg}.

$\spadesuit$ Suppose $S$ is a subset of $\C[\hfrak]$. \textbf{To ease notations, we will write $\ol{\Phi}_c(S)$ to mean $\ol{\Phi}_c(S\text{ mod }I_c)$.}

The lemma below gives a concrete characterization of $\ol{\Phi}_c$.
\begin{lem}\label{ekfk}
    Suppose $0\neq \phi\in \Lrm_c(k)$. If $k<0$, then $\ol{\Phi}_c(\phi)$ is a nonzero multiple of $\mathbf{e}^{-k}\phi$; if $k>0$, then $\ol{\Phi}_c(\phi)$ is a nonzero multiple of  $\mathbf{f}^{k}\phi$. In particular, we have that $\ol{\Phi}_c(1)$ is a nonzero multiple of $\mathbf{e}^\mathrm{d}$.
\end{lem}
\begin{proof}
    Since $\Phi_c(\ebf)=-\fbf$ and $\Phi_c(\fbf)=-\ebf$, we know that $\ol{\Phi}_c$ preserves $\langle \ebf,\fbf,\hbf\rangle $-submodules of $\Lrm_c$. On the other hand, because $\Phi_c(\hbf)=-\hbf$, for any $\phi\in \Lrm_c$ we have that $\hbf\ol{\Phi}_c(\phi)=-\ol{\Phi}_c(\hbf\cdot \phi)$.
    As a result,
    $\ol{\Phi}_c(\Lrm_c(k))=\Lrm_c(-k)$. Finally, for nonzero $\phi\in \Lrm_c(k)$ with $k<0$, there is $\mathbf{e}^{-k}\phi\in \Lrm_c(-k)\cap \langle \ebf,\fbf,\hbf\rangle\phi$, which is one-dimensional. Hence $\ebf^{-k}\phi=C\ol{\Phi}_c(\phi)$ for some nonzero constant $C$. The argument is similar for the case when $k>0$.
\end{proof}
\subsection{The algebraic filtration}\label{subsection_F_alg} In \cite[Definition 4.6]{gors}, the algebraic filtration $F^{\alg}$ is defined as:
\[F^{\alg}_i(\Lrm_c)=\bigg(\sum_{2j-k>i}(\afrak^j)(k)\bigg)^{\perp_c}.\] 
Examples can be found as \ref{example_F_ind_k/2},  \ref{3,4} and \ref{3,5}.

Below we use the Fourier transform $\ol{\Phi}_c$ to reconstruct $F^{\alg}$.
\begin{defn}
    For all $c>0$, the power filtration $F^\afrak$ on $\Lrm_c$ is defined by 
\[F^{\afrak}_i(\Lrm_c)=\ol{\Phi}_c[(\afrak^{i+1})^{\perp_c}]={\Phi}_c[(\afrak^{i+1})^{\perp_c}]\mathbf{e}^\mathrm{d}.\]
\end{defn} 
\begin{defn}\cite[3.2]{ginzburgnil}
The Kazhdan filtration $K^{F}$ on a vector space $V$ associated to an ascending filtration $F$ and a $\Z$-grading $V=\bigoplus_{k\in \Z}V(k)$ is  
\[K^{F}_i(V)=\sum_{2j+k\leq i}F_j(V)(k).\]
\end{defn}
Take the Kazhdan filtration $K^\afrak$ on $\Lrm_c$ associated to $F^\afrak$ and the $\mathbf{h}$-grading, i.e.
\[K_i^\afrak=\sum_{2j+k\leq i}\ol{\Phi}_c((\afrak^{j+1})^{\perp_c})(k).\]
\begin{lem}
    $F^{\alg}=K^\afrak$.
\end{lem}
\begin{proof} 
For fixed $i$ and $k$ we have
\[K_i^\afrak(k)=\sum_{2j+k\leq i}\ol{\Phi}_c\left(\left(\afrak^{j+1}\right)^{\perp_c}\right)(k)=\ol{\Phi}_c\left(\left(\afrak^{\lceil \frac{i-k+1}{2}\rceil}\right)^{\perp_c}\right)(k),\]
\[F_i^{\alg}(k)=\bigg(\sum_{\ell<2j-i}(\afrak^j)(\ell)\bigg)^{\perp_c}(k)=\left(\afrak^{\lceil \frac{i+k+1}{2}\rceil}\right)^{\perp_c}(k).\]

Take $P\in \Phi_c[(\afrak^{\lceil \frac{i-k+1}{2}\rceil})^{\perp_c}]$, a polynomial in $y_1-y_2,\dots, y_{n-1}-y_n$ of degree $\mathrm{d}-k$. Take $\phi\in \afrak^{\lceil \frac{i+k+1}{2}\rceil}$ of degree $\mathrm{d}+k$ so that $\phi\in \Lrm_c(k)$. Then
\[(\phi,P\mathbf{e}^\mathrm{d})_c=(\Phi_c(P),\Phi_c(\phi)\mathbf{e}^\mathrm{d})_c. \]
As a result $F_i^\alg(k)=K_i^{\afrak}(k)$ if and only if $\ol{\Phi}_c(\afrak^{\lceil \frac{i+k+1}{2}\rceil})(-k)=\afrak^{\lceil \frac{i-k+1}{2}\rceil}(-k)$. After change of variables, this latter condition is equivalent to $\ol{\Phi}_c(\afrak^u)(v)=\afrak^{u+v}(v)$.

Now by Lemma \ref{ekfk},  $\ol{\Phi}_c(\afrak^u)(v)=\ol{\Phi}_c(\afrak^u(-v))$, which equals $\ebf^{v} (\afrak^u(-v))=(\ebf^{v} \afrak^u)(v)$ when $v>0$ and equals $\fbf^{-v} (\afrak^u(-v))=(\fbf^{-v} \afrak^u)(v)$ when $v<0$. Since $(\ebf^{v} \afrak^u)(v)$ and $\fbf^{-v} (\afrak^u(-v))=(\fbf^{-v} \afrak^u)(v)$ both lie in $\afrak^{u+v}(v)$, we conclude \[\Phi_c(\afrak^u)(v)\subset \afrak^{u+v}(k),\] from which we also have \[\ol{\Phi}_c(\afrak^{u+v}(v))=(\ol{\Phi}_c(\afrak^{u+v}))(-v)\subset \afrak^{u}(-v).\]
Now applying $\ol{\Phi}_c$ gives us $\afrak^{u+v}(v)\subset \ol{\Phi}_c(\afrak^{u}(-v))=\ol{\Phi}_c(\afrak^u)(v)$. Therefore $\ol{\Phi}_c(\afrak^u)(v)=\afrak^{u+v}(v)$ and $F^{\alg}=K^\afrak$.
\end{proof}

\subsection{The inductive filtration}\label{subsection_F_ind}
Denote the idempotent associated to the trivial representation of $S_n$, resp. $S_m$ by $\e_n$, resp. $\e_m$. We have the following isomorphisms:
\begin{align}
&\text{when }m,n>1,\quad &\e_n\Lrm_{\frac{m}{n}}&\cong \e_m\Lrm_{\frac{n}{m}}\quad&\text{\cite[8.2]{cee}};\label{1}\\
    &\text{when $c>1$,\quad}&\Lrm_c&\cong\Hrm_c  \e_-\otimes_{\e\Hrm_{c-1}\e}\e\Lrm_{c-1}\quad&\text{\cite[Lemma 4.7]{beg}}.\label{3}
\end{align}
Here $\Hrm_c  \e_-$ is a right $\e\Hrm_{c-1}\e$-module via the identification $\e_-\Hrm_c\e_-\cong \delta\e\Hrm_{c-1}\e\delta^{-1}$ (\cite[Proposition 4.1]{beg}). The isomorphism (\ref{3}) implies an isomorphism: \begin{align}
    \label{2}
\e_-\Lrm_c\cong \delta \e\Lrm_{c-1}\quad \text{when $c>1$.}
\end{align} 
Using (\ref{1}) and (\ref{3}), we can inductively define two filtrations on $\Lrm_c$ when $c>1$ and on $\e \Lrm_c$ when $c>0$.

First we define a partial order on the positive rational numbers satisfying the following conditions: for all coprime integers $m,n$, 
\begin{align}
    \label{order}
\frac{m}{n}\prec \frac{m+n}{n}; \text{ if $n<m$, } \frac{m}{n}\prec \frac{n}{m}.
\end{align}

One can follow the Euclidean algorithm to go from $c=\frac{m}{n}$ to $\frac{1}{n'}$ for some integer $n'>1$ through a chain of rational numbers decreasing under the order $(\Q_{>0},\prec)$. For example, if $c=\frac{13}{5}$, then we have $\frac{13}{5}\succ \frac{3}{5}\succ \frac{5}{3}\succ \frac{2}{3}\succ \frac{3}{2}\succ \frac{1}{2}$.
\begin{defn}\cite[Theorem 4.1]{gors}
For the base case when $m=1$, $F^\ind$ and $F^{\ul{\ind}}$ are defined trivially on $\Lrm_c$:  
\[0=F^{\star}_{-1}\e\Lrm_{\frac{1}{n}}\subset F^{\star}_0\e\Lrm_{\frac{1}{n}}=\e\Lrm_{\frac{1}{n}}=\Lrm_{\frac{1}{n}}.\] 
where $\star=\ind$ or $\ul{\ind}$. 
Next, $F^{\ind}$ and $F^{\ul{\ind}}$ are defined inductively under the order $(\Q_{>0},\prec)$ using the isomorphisms (\ref{1}) and (\ref{3}).
\begin{itemize}
    \item When the filtration on the right hand side of (\ref{3}) is the tensor product filtration where $\Hrm_c$ is endowed with the Bernstein filtration $F^{\mathrm{Bern}}$ such that $\deg y=\deg x=1$ and $\deg w=0$ we obtain $F^{\ind}$.
    \item When the filtration on the right hand side of (\ref{3}) is the tensor product filtration where $\Hrm_c$ is endowed with the order filtration $F^{\ord}$ such that $\deg y=1$ and $\deg x=\deg w=0$ we obtain $F^{\ul{\ind}}$.  
\end{itemize}
\begin{ex}\label{example_F_ind_k/2}
    Take $n=2$ and $c=k+\frac{1}{2}$ for $k\in \Z_{\geq 0}$. Writing $\xi=x_1-x_2$, we have that $\Lrm_c=\C[\xi]/\left(\xi^{2k+1}\right)$ as in Example \ref{ex_L_k/2} and $\left (\afrak^j\right)^{\perp_c} =\sum_{\ell<2j}\C \xi^\ell $. Figure \ref{k/2} illustrates the filtrations $F^\alg$ and $F^\afrak$, there is $\afrak^\perp$ or equivalently, $F^\ind$ and $F^{\ul{\ind}}$. 
         \newcommand{\width}{2}
         \begin{figure}\label{4/3}
         	\centering
         	\begin{tikzpicture}
         		\node at (-1*\width,-1) {\dots};
         		\node at (0*\width,-1) {$k-4$};
         		\node at (1*\width,-1) {$k-2$};
         		\node at (2*\width,-1) {$k$};
         		\node at (-1*\width,1) {\dots};
         		\node at (0*\width,1) {$\xi^{2k-4}$};
         		\node at (1*\width,1) {$\xi^{2k-2}$};
         		\node at (2*\width,1) {$\xi^{2k}$};
         		\node at (-1.5*\width,2) {\dots};
         		\node at (-0.5*\width,2) {$\xi^{2k-5}$};
         		\node at (0.5*\width,2) {$\xi^{2k-3}$};
         		\node at (1.5*\width,2) {$\xi^{2k-1}$};
         		\draw (-1.6*\width,0.7) --  (2.2*\width,0.7)-- (3.5*\width, 1.5);
         		\draw  (-1.8*\width,1.7) --  (2.2*\width,1.7)-- (3.2*\width, 2.3);
         		\draw[dashed, opacity=.5] (-1.2*\width,2.3) -- (.28*\width,-.4) -- (3.5*\width,-.4);
         		\draw[dashed, opacity=.5] (-0.1*\width,2.3) -- (1.2*\width,-.3) -- (3.5*\width,-.3);
         		\draw[dashed, opacity=.5] (0.9*\width,2.3) -- (2.1*\width,-.2) -- (3.5*\width,-.2);
         		\node at (3.0*\width,2.5) {$\color{blue}{F^{\mathrm{alg}}_{k-1}}$};
         		\node at (3.3*\width,1.8) {$\color{blue}{F^{\mathrm{alg}}_{k}}$};
         		\node at (2.4*\width,0.05) {$\color{red}{F^{\mathfrak{a}}_{0}}$};
         		\node at (1.35*\width,-.05) {$\color{red}{F^{\mathfrak{a}}_1}$};
         		\node at (0.55*\width,-.15) {$\color{red}{F^{\mathfrak{a}}_2}$};
         	\end{tikzpicture}
         	\caption{Filtrations in Example \ref{example_F_ind_k/2}. Numbers at the bottom indicate $\hbf$-weights}
         	\label{k/2}
         \end{figure}
\end{ex}
\end{defn}
Take the Kazhdan filtration $K^{\ul{\ind}}$ associated to $F^{\ul{\ind}}$ and the $\mathbf{h}$-grading. 
\begin{lem}
    $F^\ind=K^{\ul{\ind}}$.
\end{lem}
\begin{proof}
It is sufficient to show that the Bernstein filtration $F^{\text{Bern}}$ on $\Dcal(\ol{\hfrak}_{r})$ (containing $\Hrm_c$) defined by $\deg(x_i)=\deg({\partial_{x_i}})=1$ is the Kazhdan filtration associated to the the $\mathbf{h}$-grading and the order filtration $F^{\text{ord}}$ defined by $\deg(x_i)=0$ and $\deg({\partial_{x_i}})=1$. 

By definition, $F_i^{\text{ord}}(\Dcal(\ol{\hfrak}_{r}))=\sum_{|\alpha|\leq i}\Ocal_{\ol{\hfrak}_{r}}\partial_{x}^\alpha$ with $|\alpha|=\sum_i \alpha_i$ whose associated Kazhdan filtration is
\[K_i^{\text{ord}}(\Dcal(\ol{\hfrak}_{r}))=\sum_{2j+k\leq i}\sum_{|\alpha|\leq j}\Ocal_{\ol{\hfrak}_{r}}\partial_{x}^\alpha(k).\]
Note that the $\mathbf{h}$-grading of $x_i$ is $1$ while that of $\partial_{x_i}$ is $-1$. Therefore 
\[K_i^{\text{ord}}(\Dcal(\ol{\hfrak}_{r}))=\sum_{2j+|\beta|-|\alpha|\leq i}\sum_{|\alpha|\leq j}\xi_{\alpha,\beta}x^\beta\partial_{x}^\alpha=\sum_{|\beta|+|\alpha|\leq i}\xi_{\alpha,\beta}x^\beta\partial_{x}^\alpha.\]
where $\xi_{\alpha,\beta}\in \C$. This is exactly $F^{\text{Bern}}$ and the lemma is proved.
\end{proof}
\begin{coro}\label{kazhdan}
    The equality $F^\alg=F^{\ind}$ is equivalent to the equality $F^\afrak=F^{\ul{\ind}}$.
\end{coro}
\subsection{Filtrations and the $\mathfrak{sl}_2$-triple}So far we have transformed the statement of the main Theorem \ref{gorsthm} into an equivalent but more approachable equality. In preparation of proving the identity $F^\afrak=F^{\ul{\ind}}$, in this subsection we study how the filtrations interact with the $\mathfrak{sl}_2$-triple $\{\ebf,\fbf,\hbf\}$.

Let $\C[\hfrak]_j^W$ be the subspace of 
$\C[\hfrak]^W$ spanned by $p_{i_1}p_{i_2}\cdots p_{i_j}$ for $2\leq i_1,\cdots, i_j\leq n$.
\begin{lem}\label{power}
$\phi \in (\afrak^i)^{\perp_c}
\subset \Lrm_c$ if and only if for all $\psi \in \C[\hfrak]^W_i$,  $\Phi_c(\psi)\phi=0$.
\end{lem}
\begin{proof}
``$\Leftarrow$" is clear. To show ``$\Rightarrow$", assume $\phi\in (\afrak^i)^{\perp_c}$ but $\Phi_c(\psi)\phi\neq 0$ for some $\psi \in \C[\hfrak]^W_i$. Then since $(-,-)_c$ is non-degenerate on $\Lrm_c$, there exists some $\xi\in \Lrm_c$ so that $(\xi, \Phi_c(\psi)\phi)_c=(\xi\psi,\phi)_c\neq 0$ is a nonzero constant. This contradicts the assumption that $\phi\in (\afrak^i)^{\perp_c}$ as $\xi\psi\in \afrak^i$.
\end{proof}

\begin{lem}\label{ef}
\begin{enumerate}[(1)]
    \item $\mathbf{f}\cdot F^{\afrak}_i\Lrm_c\subset F^{\afrak}_{i+1}\Lrm_c$, for $c>0$
    \item $\mathbf{f}\cdot F^{\ul{\ind}}_i\e\Lrm_c\subset F^{\ul{\ul{\ind}}}_{i+1}\e\Lrm_c$ for $c>0$ and $\mathbf{f}\cdot F^{\ul{\ind}}_i\Lrm_c\subset F^{\ul{\ind}}_{i+1}\Lrm_c$ for $c>1$.
    \item $\mathbf{e}\cdot F^{\afrak}_i\Lrm_c\subset F^{\afrak}_{i-1}\Lrm_c$, for $c>0$
    \item $\mathbf{e}\cdot F^{\ul{\ind}}_i\e\Lrm_c\subset F^{\ul{\ind}}_{i-1}\e\Lrm_c$ for $c>0$ and $\mathbf{e}\cdot F^{\ul{\ind}}_i\Lrm_c\subset F^{\ul{\ind}}_{i-1}\Lrm_c$ for $c>1$.
    \item  Assume $v$ is a highest weight vector with respect to the $\mathfrak{sl}_2$ triple $\{\ebf,\fbf,\hbf\}$ and $v\in F_j^{\ul{\ind}}\Lrm_c\setminus F_{j-1}^{\ul{\ind}} \Lrm_c$. Then for $ i\geq 0$ such that $\fbf^iv\neq 0$, we have $\fbf^iv\in F_{j+i}^{\ul{\ind}}\Lrm_c\setminus F_{j+i-1}^{\ul{\ind}}\Lrm_c$. The same holds for $F^\afrak$ in place $F^{\ul{\ind}}$.
\end{enumerate}

\end{lem}
\begin{proof}
We will only prove (1) and (2). (3) and (4) can be shown by an analogous argument. (5) follows from (1)-(4).

By definition, (1) is to say $\fbf\cdot \ol{\Phi}_c((\afrak^{i+1})^{\perp_c})\subset \ol{\Phi}_c((\afrak^{i+2})^{\perp_c})$, or equivalently $\ebf \cdot (\afrak^{i+1})^{\perp_c}\subset (\afrak^{i+2})^{\perp_c}$. By Lemma \ref{power} and equation (\ref{xy}): $((x_i-x_j)\phi,\psi)_c=(\phi,(y_i-y_j)\psi)_c$, it suffices to show $\fbf\cdot \C[\hfrak]^W_{i+2}\subset \C[\hfrak]^W_{\geq i+1}$.

Recall that $\fbf=-\frac{1}{2n}\sum_{i<j} (y_i-y_j)^2$. We write  $\wt{\nabla}f=(\frac{\partial f}{\partial_{x_i}}-\frac{\partial f}{\partial_{x_j}})_{i<j}$
and use \cite[Corollary 5.3]{gorskyarc} to compute that for $2\leq j_1,\cdots j_{i+2}\leq n$
\begin{align*}
-2n\fbf\cdot (p_{j_1}\cdots p_{j_{i+2}})=&\left(-2n\fbf\cdot p_{j_1}\right)p_{j_2}\cdots p_{j_{i+2}}+\cdots +p_{j_1}p_{j_2}\cdots \left( -2n\fbf \cdot p_{j_{i+2}}\right)\\
&+(\wt{\nabla} p_{j_1}\cdot \wt{\nabla} p_{j_2})\cdots p_{j_{i+2}}+p_{j_1}\cdots (\wt{\nabla} p_{j_i}\cdot \wt{\nabla} p_{j_{i+2}})
\end{align*}
which does lie in $\C[\hfrak]^W_{\geq i+1}$.

Next, we prove (2) by induction. 
Trivially, part (2) of the lemma holds for $\e\Lrm_{1/n}$ for all $n>1$. Now for $c=m/n$ where $m>1$ we assume the statement holds for all $(m',n')\prec (m,n)$. For $\e\Lrm_c$ we may assume $c>1$ otherwise we take $\e\Lrm_{1/c}$ given axiom (1) for $F^{\ind}$. Thus we only need to prove the statement for $\Lrm_c$.

Take $\sum \xi_k\otimes \eta_k\in \Hrm_c \delta \e_-\otimes_{\e_-\Hrm_c\e_-}\e\Lrm_{c-1}$ so that $\xi_k\in F_\alpha \Hrm_c\delta \e_-$ and $\eta_k\in F^{\ind}_\beta \e\Lrm_{c-1}$ and $\alpha+\beta\leq i$. Apply $\Phi_c(p_2)$ and we obtain
\[\sum \Phi_c(p_2)\xi_i\otimes \eta_i=\sum [\Phi_c(p_2),\xi_i]\otimes \eta_i + \xi_i\otimes \Phi_c(p_2)\eta_i.\]
Here $[\Phi_c(p_2),\xi_i]\in F_{\alpha+1}$ and $\Phi_c(p_2)\eta_i\in F^{\ind}_{\alpha+1}$ by the inductive hypothesis. Hence (2) of the lemma follows.
\end{proof}
\begin{coro}\label{sl2}
\begin{enumerate}[(1)]
    \item $\ebf(F_i^\alg\Lrm_c)\subset F_i^\alg\Lrm_c$ and $\ebf(F_i^{\ind}\Lrm_c)\subset F_i^{\ind}\Lrm_c$.
     \item $\fbf(F_i^\alg\Lrm_c)\subset F_i^\alg\Lrm_c$ and $\fbf(F_i^{\ind}\Lrm_c)\subset F_i^{\ind}\Lrm_c$.
    \item $\ol{\Phi}_c(F_i^\alg\Lrm_c)= F_i^\alg\Lrm_c$ and $\ol{\Phi}_c(F_i^{\ind}\Lrm_c)=F_i^{\ind}\Lrm_c$.
    \item Assume $v$ is a highest weight vector under the $\mathfrak{sl}_2$ triple $\{\ebf,\fbf,\hbf\}$ and $v\in F_j^{\ind}\Lrm_c\setminus F_{j-1}^{\ind} \Lrm_c$. Then for $j\geq 0$ such that $\fbf^jv\neq 0$, $\fbf^jv\in F_{j}^{\ind}\Lrm_c\setminus F_{j-1}^{\ind}\Lrm_c$. The same holds for $F^\afrak$ in place $F^{\ind}$.
\end{enumerate}
\end{coro}
\begin{proof}
    Notice that $\mathbf{e}\cdot \Lrm_c(k)\subset \Lrm_c(k+2)$ while $\mathbf{f}\cdot \Lrm_c\subset \Lrm_c(k-2)$. Hence (1) and (2) of the corollary follow from Lemma \ref{ef} and the fact that $F^\alg$ is the Kazhdan filtrations associated to $F^\afrak$ and $F^{{\ind}}$ is the Kazhdan filtrations associated to $F^{\ul{\ind}}$. (3) follows from (1) and (2) plus Lemma \ref{ekfk}. (4) follows from (1),(2) and (3).
\end{proof}
\section{Some first relations between $F^\afrak$ and $F^{\ul{\ind}}$}\label{section_first_relations}
\subsection{$F^\afrak$ is compatible with (\ref{1}) and (\ref{2})}It is discussed in \cite[Theorem 4.8]{gors} that $F^\alg$ is compatible with the isomorphisms (\ref{1}) and (\ref{2}). The power filtration $F^\afrak$ is also compatible with (\ref{1}) and (\ref{2}): According to \cite[8.2]{cee} the isomorphism $\e\Lrm_{\frac{m}{n}}\cong \e\Lrm_{\frac{n}{m}}$ (\ref{1}) can be defined by (up to scalars) \[p_i(x_1,\dots, x_n)\mapsto p_i(x_1,\dots, x_m),\] 
which is clearly a filtered homomorphism under $F^\afrak$. Also, the following lemma shows that $F^\afrak$ is compatible with (\ref{2}):
\begin{lem}\label{delta1}
    $\delta \left((\afrak^i)^{\perp_c}\right)^W\subset (\afrak^i)^{\perp_c}$ and  $\Phi_c(\delta) \left((\afrak^i)^\perp\right)^{\sign}\subset (\afrak^i)^{\perp_c}$.
\end{lem}
\begin{proof}
We will only show that $\delta \left((\afrak^i)^{\perp_c}\right)^W\subset (\afrak^i)^{\perp_c}$, and the proof of the other statement is similar. Take $\phi\in \afrak^i$ and $\psi\in \left((\afrak^i)^{\perp_c}\right)^W$ such that $(\phi,\delta\psi)_c\neq 0$. By Lemma \ref{isotypic}, we may assume that $\phi=\delta\eta$ for some $\eta\in(\afrak^i)^W$. Using \cite[Corollary 4.5]{dunklsingular}, we compute that \begin{align}
    \label{dunkl}
0\neq (\delta\eta,\delta\psi)_c=(\delta,\delta)_c(\eta,\psi)_{c-1}\end{align} This contradicts the assumption that $\psi\in (\afrak^i)^{\perp_c}$. 
\end{proof}
\subsection{The containment ``$F^{\ul{\ind}}\subset F^\afrak$"}In view of Corollary \ref{kazhdan}, Theorem \ref{gorsthm} is implied by the following statement:
\begin{thm}\label{thm}
      Suppose $c=\frac{m}{n}$ for positive $m$ coprime to $n$. 
      
      We have that 
      \[\begin{cases}
 F_i^{\ul{\ind}}\Lrm_{\frac{m}{n}}=F^\afrak\Lrm_{\frac{m}{n}} &\forall i\geq 0\text{ when }m>n\\
 F_i^{\ul{\ind}}\e\Lrm_\frac{m}{n}=F^\afrak \e\Lrm_\frac{m}{n} &\forall i\geq 0
 \end{cases}\]
\end{thm}
The containment in one direction is immediate:
\begin{lem}\label{containment}
    Under the assumption of Theorem \ref{thm},   \[\begin{cases}
 F_i^{\ul{\ind}}\Lrm_{\frac{m}{n}}\subset F^\afrak\Lrm_{\frac{m}{n}} &\forall i\geq 0\text{ when }m>n\\
 F_i^{\ul{\ind}}\e\Lrm_\frac{m}{n}\subset F^\afrak \e\Lrm_\frac{m}{n} &\forall i\geq 0
 \end{cases}\]
\end{lem}
\begin{proof}
   We do induction on $c$ with respect to the partial order (\ref{order}). The base case when $c=\frac{1}{n}$ is clear. Suppose $c=\frac{m}{n}$ for $m>1$. Assume $F_i^{\ul{\ind}}\Lrm_{c'}\subset F^\afrak_i\Lrm_{c'}$ holds for all $i\geq 0$ and for all $c'\prec c$ such that $c'>1$; assume also that $F_i^{\ul{\ind}}\e\Lrm_\frac{m}{n}\subset F^\afrak \e\Lrm_\frac{m}{n}$ for all $i\geq 0$ and for all $c'\prec c$. 
   
   When $c=\frac{m}{n}<1$, we only need to verify the inclusion for $\e\Lrm_\frac{m}{n}$. Since both $F^\afrak$ and $F^{\ul{\ind}}$ are compatible with the isomorphism (\ref{1}): $\e_n\Lrm_{\frac{m}{n}}\cong \e_m\Lrm_{\frac{n}{m}}$, the inclusion can be restated in terms of $\e_m\Lrm_\frac{n}{m}$. Thus, we may assume $c>1$.
   
   Take $\sum_k \xi_k\otimes \eta_k\in \Hrm_c \delta \e_-\otimes_{\e_-\Hrm_c\e_-}\e\Lrm_{c-1}$ so that $\xi_k\in F^{\text{ord}}_\alpha \Hrm_c\delta \e_-$ and $\eta_k\in F^{\ul{\ind}}_\beta \e\Lrm_{c-1}$ and $\alpha+\beta\leq i$. We shall show $ (\xi_k\otimes \eta_k)\perp_c \ol{\Phi}_c(\afrak^{i+1})$.

    By Lemma \ref{power}, it is sufficient to show that $p_{j_{i+1}}\cdots p_{j_1}(\xi_k\otimes \eta_k)=0$ for $2\leq j_1,\cdots j_{i+1}\leq n$. For this, we compute that
    \begin{align*}
        &p_{j_{i+1}}\cdots p_{j_1}(\xi_k\otimes \eta_k)\\
        =&p_{j_{i+1}}\cdots p_{j_2} [p_{j_{1}}, \xi_k]\otimes \eta_k+ p_{j_{i+1}}\cdots p_{j_2}\xi_k \otimes p_{j_{1}}\eta_k=\cdots \\
        =&[p_{j_{i+1}},[\cdots,[ p_{j_{\alpha+1}},[p_{j_\alpha},[p_{j_{\alpha-1}},[\cdots[p_{j_1}, \xi_k]\cdots ]\otimes \eta_k\\
        &+[p_{j_{i+1}},[\cdots, [p_{j_{\alpha+1}},[p_{j_\alpha},[p_{j_{\alpha-1}},[\cdots[p_{j_2}, \xi_k]\cdots ]\otimes  p_{j_1}\eta_k+\cdots\\
        &+\xi_k\otimes p_{j_{i+1}}\cdots p_{j_1}\eta_k.
        \end{align*}
Notice that \[[p_{j_{i+1}}[\cdots [p_{j_{\alpha+1}},[p_{j_\alpha},[p_{j_{\alpha-1}},[\cdots[p_{j_{\ell+1}}, \xi_k]\cdots ]\in F_{\alpha-i+\ell-1}\] when $\alpha-i+\ell-1\geq 0$ and equals $0$ otherwise.

On the other hand, by the induction hypothesis,  $\eta_k  \perp \Phi_c(\afrak^{\beta+1})\sigma_{c-1}$. As a result, $p_{j_{\ell}}\cdots p_{j_1}\eta_k \in F_{\beta-\ell}^{\afrak} \e\Lrm_{c-1}$ when $\ell\leq \beta$ and equals $0$ otherwise. 

Since $\alpha+\beta\leq i$, we see that either \[[p_{j_{i+1}}[\cdots [p_{j_{\alpha+1}},[p_{j_\alpha},[p_{j_{\alpha-1}},[\cdots[p_{j_{\ell+1}}, \xi_k]\cdots ]=0\] or  $p_{j_{\ell}}\cdots p_{j_1}\eta_k =0$ and the lemma follows.    
\end{proof}
\begin{rem}
This gives a different proof of \cite[Theorem 4.8]{gors}.  
\end{rem}
\subsection{Identity on $F_0$}To show the opposite containment ``$F_i^{\ul{\ind}}\Lrm_c\supset F^\afrak_i\Lrm_c$", our idea is to do induction on $i$ and $c$. For the base case:
\begin{lem}\label{F0}
When $c>1$, $F_0^{\ul{\ind}}\Lrm_c=F_0^{\afrak}\Lrm_c$.
\end{lem}
\begin{proof}
    It suffices to show $F_0^{\ul{\ind}}\Lrm_c\supset F_0^{\afrak}\Lrm_c$. Under the isomorphism (\ref{3}),  
   \begin{align*}
F_0^{\afrak}\Lrm_c{=}\ol{\Phi}_c(\afrak^{\perp_c})\stackrel{\text{Lemma }\ref{perp}}{=}\ol{\Phi}_c(\Hscr_c)=\Phi_c(\C[y_1,\cdots, y_n]\delta)\mathbf{e}^\mathrm{d}= \C[x_1,\cdots,x_n]\Phi_c(\delta)\ebf^\mathrm{d}
\end{align*}
    is mapped to \begin{align*}
    &\im\big( \C[x_1,\cdots,x_n]\e_-\otimes \ebf^{\mathrm{d}-\frac{n(n-1)}{2}} \Hrm_c  \e_-\otimes_{\e\Hrm_{c-1}\e}\e\Lrm_{c-1}\big)\\
    \subset\  &\im\big(F_0^\ord(\Hrm_c\e_-)\otimes F_0^{\ul{\ind}}\e\Lrm_{c-1}\to \Hrm_c  \e_-\otimes_{\e\Hrm_{c-1}\e}\e\Lrm_{c-1}\big)
    =F^{\ul{\ind}}_0\Lrm_c.\qedhere
    \end{align*}
\end{proof}
\begin{rem}
    This recovers \cite[Corollary 4.11]{gors} from a different point of view.
    \end{rem}
\begin{ex}\label{3,4}(See Figure \ref{4/3})
 Consider $\Lrm_{\frac{4}{3}}$ whose dimension is $4^{3-1}=16$. 
 
 We compute $F^\afrak, F^{\ul{\ind}}, F^\alg$ on $\Lrm_{\frac{4}{3}}$. In this case, $\Hscr_c=\Hscr_c^{\triv}\oplus (\Hscr_c^{\st})^{\oplus 2}\oplus \Hscr_c^\sign$. Inside,  \[(\Hscr_c^{\st})^{\oplus 2}=\langle \xi_1,\xi_2\rangle\oplus \langle \alpha_1,\alpha_2\rangle\] where $\xi_i=x_i-x_{i+1}$ and  $\alpha_i=(y_i-y_{i+1})\delta$. 
 By Lemma \ref{F0},  we have \[F_0^{\ul{\ind}}\Lrm_c=F_0^\afrak\Lrm_c=\langle \ol{\Phi}_c(1),\ol{\Phi}_c(\xi_i),\ol{\Phi}_c(\alpha_i),\ol{\Phi}_c(\delta)\rangle.\]  Since $\ol{\Phi}_c(1),\ol{\Phi}_c(\xi_i),\ol{\Phi}_c(\alpha_i),\ol{\Phi}_c(\delta)$ are all highest weight vectors under the $\mathfrak{sl}_2$ triple $\{\ebf,\fbf,\hbf\}$,  Lemma \ref{ef} tells us that $\fbf^i(v)\in F_i^{\ul{\ind}}\Lrm_c\cap F_i^\afrak\Lrm_c$ for any $v\in F_0^\afrak\Lrm_c$.

It remains to determine where $p_3$ belongs to. Since $p_3\in(\afrak^2)^{\perp_c}\setminus \Hscr_c$, one has that $\ol{\Phi}_c(p_3)\in F^\afrak_1\Lrm_c$. On the other hand, by Lemma \ref{h}, $p_3=-\frac{1}{9}\sum_{i<j} (x_i-x_j)(y_i-y_j)p_3$. Because $(y_i-y_j)p_3\in \Hscr_c$, we conclude that \[\ol{\Phi}_c(p_3)=-\frac{1}{9}\sum_{i<j} (y_i-y_j) \ol{\Phi}_c((y_i-y_j)p_3)\in F_1^{\ul{\ind}}\Lrm_c.\]
\end{ex}
\newcommand{\width}{2}
\begin{figure}\label{4/3}
\centering
\begin{tikzpicture}
	\foreach \xy in {-3,...,3}{
		\node at (\width*\xy,-1) {\xy};
	}
	\node at (-3*\width,1) {1};
	\node at (-1*\width,1) {$p_2$};
	\node at (0*\width,1) {$p_3$};
	\node at (1*\width,1) {$p_2^2$};
	\node at (3*\width,1) {$p_2^3$};
	\node at (-2*\width,2) {$\xi_i$};
	\node at (0,2) {$\xi_ip_2$};
	\node at (2*\width,2) {$\xi_ip_2^2$};
	\node at (-\width,3) {$\alpha_i$};
	\node at (\width,3) {$\alpha_ip_2$};
	\node at (0,4) {$\delta$};
	\draw (-3.5*\width,3) -- (-2.5*\width,1.5) -- (-.6*\width,1.5) -- (0,.6) -- (.6*\width,1.5) --(2.5*\width,1.5) -- (3.5*\width,3);
	\draw (-2.5*\width,4) -- (-1.5*\width,2.5) -- (1.5*\width,2.5) -- (2.5*\width,4);
	\draw (-1.1*\width,4.9) -- (0,3.5) -- (1.1*\width,4.9);
	\draw[dashed, opacity=.5] (-3.3*\width,3.3) -- (-.8*\width,-.4) -- (3.5*\width,-.4);
	\draw[dashed, opacity=.5] (-2.4*\width,4.2) -- (.5*\width,-.3) -- (3.5*\width,-.3);
	\draw[dashed, opacity=.5] (-1.2*\width,4.8) -- (2.7*\width,-.2) -- (3.5*\width,-.2);
	\node at (.9*\width,5) {$\color{blue}{F^{\mathrm{alg}}_0}$};
	\node at (2.3*\width,4.1) {$\color{blue}{F^{\mathrm{alg}}_1}$};
	\node at (3.3*\width,3.1) {$\color{blue}{F^{\mathrm{alg}}_2}$};
	\node at (3.3*\width,0.05) {$\color{red}{F^{\mathfrak{a}}_0}$};
	\node at (.75*\width,-.05) {$\color{red}{F^{\mathfrak{a}}_1}$};
	\node at (-.55*\width,-.15) {$\color{red}{F^{\mathfrak{a}}_2}$};
\end{tikzpicture}
\caption{Filtrations from example \ref{3,4}. Numbers at the bottom indicate  $\hbf$-weights}
\label{4/3}
\end{figure}
\section{The orthogonal complement of $I_{c-1}^W$}\label{subsection_orthogonal_complement}
\subsection{Orthogonal lifts} Let $\C[\hfrak]^W_{\leq j}=\bigoplus_{\ell=0}^j \C[\hfrak]^W_\ell$. The product $\Hscr_c \cdot \C[\hfrak]^W_{\leq i}$ denotes the vector space generated by all $h \psi$ with $h\in \Hscr_c$ and $\psi\in \C[\hfrak]^W_{\leq i}$.

$\spadesuit$ Since $\Hscr_c \cdot\C[\hfrak]^W_{\leq i} \cong \C[\hfrak]/\afrak^{i+1}$, there is a \textbf{unique orthogonal lift} $\ul{h\psi}\in \Lrm_c$ so that $\ul{h\psi}\in (\afrak^{i+1})^{\perp_c}$ and $\ul{h\psi}\equiv h\psi$ modulo $\afrak^{i+1}$. \textbf{In the sequel, all polynomials will be homogeneous.}

\subsection{The orthogonal complement of $I_{c-1}^W$}
We would like to approach Theorem \ref{thm} by considering smaller subspaces of $\C[\hfrak]$ as follows. Take $(\afrak^{j+1})^{\perp_c}_{j}:=(\afrak^{j+1})^{\perp_c}\cap \afrak^{j}$. When $j\neq j'$ \[ (\afrak^{j+1})_{j}^{\perp_c}\perp_c(\afrak^{j'+1})_{j'}^{\perp_c}.\]

 Define a subspace of $(\afrak^{j+1})_{j}^{\perp_c}$:
 \[\mathcal{S}_{j,\perp_c}=\langle \ul{h\psi}\in  (\afrak^{j+1})_{j}^{\perp_c}|h\in \Hscr_c, \psi\in (I^{\perp_c}_{c-1})^W\rangle.\]
 The following lemma shows that we can prove $F^{\ul{\ind}}=F^\afrak$ on $\ol{\Phi}_c(\mathcal{S}_{j,\perp_c})  $ by induction on $c$.
\begin{lem}\label{notinI}
For any integer $j\geq 0$, suppose $F_j^{\ul{\ind}}\e\Lrm_{c-1}=F_j^\afrak\e\Lrm_{c-1}$. Then $\ol{\Phi}_c(\mathcal{S}_{j,\perp_c})\subset  F^{\ul{\ind}}_j\Lrm_c$.
\end{lem}
\begin{proof}
We can go through the Gram--Schmidt process to obtain an orthogonal basis $\{\wh{\psi}\}$ of \[\big((I_{c-1})^{\perp_c}\cap (\afrak^{j+1})_{j}^{\perp_c}\big)^W\text{ mod }I_c\] under the nondegenerate form $(-,-)_{c-1}$ on this space. Also, we take an orthogonal basis $\{\wh{h}\}$ of $\Hscr_c$ such that $\wh{h}=\wh{P}\delta$ with $\wh{P}\in \C[y_1-y_2,\dots, y_{n-1}-y_n]$ depending on $\wh{h}$. Then $\{\ol{\Phi}_c(\wh{h}\wh{\psi})\}$ forms a basis of $\ol{\Phi}_c(\mathcal{S}_{j,\perp_c})$.

For basis elements $\wh{\psi}$ and $\wh{P}\delta$, we claim the following:
\begin{enumerate}[(a)]
    \item $\wh{P}(\delta \wh{\psi})\in (\afrak^{j+1})^{\perp_c}$.
    \item $\ol{\Phi}_c(\wh{P}(\delta \wh{\psi}))\in F^{\ul{\ind}}_j \Lrm_c$.
    \item For any pair $(\wh{h'}, \wh{\psi'})$, we have that $(\wh{P}(\delta \wh{\psi}), \wh{h'}\wh{\psi'})_c\neq 0$ if and only if $\wh{h'}= \wh{h}$ and $\wh{\psi'}=\wh{\psi}$.
\end{enumerate}

Here (a) holds because $\delta\wh{\psi}\in (\afrak^{j+1})^{\perp_c}$ by Lemma \ref{delta1} and hence $\wh{P}(\delta \wh{\psi})\in (\afrak^{j+1})^{\perp_c}$ as well.

For (b), by Lemma \ref{ekfk}, $\ol{\Phi}_c(\delta \wh{\psi})={\Phi}_c(\delta \wh{\psi})\beta_c$ for some nonzero $\beta_c$ in $\Lrm_c$ of weight $\mathrm{d}$. The element ${\Phi}_c(\delta \wh{\psi})\beta_c$ is mapped to $\Phi_{c-1}(\wh{\psi})\beta_{c-1}$ for some nonzero highest weight vector $\beta_{c-1}$ in $\Lrm_{c-1}$ under the isomorphism $\e_-\Lrm_c\cong \delta\e\Lrm_{c-1}$. The element $\Phi_{c-1}(\wh{\psi})\beta_{c-1}$ lies in $F_j^\afrak \e\Lrm_{c-1}$ and hence belongs to $F^{\ul{\ind}}_j \Lrm_{c-1}$ by the assumption of the lemma.  Therefore $\ol{\Phi}_c(\wh{P}(\delta \wh{\psi}))=\Phi_c(\wh{P})\ol{\Phi}_c(\delta \wh{\psi})$ also lies in $F^{\ul{\ind}}_j \Lrm_{c-1}$.

To show (c), notice that \[\Phi_c(\delta \wh{\psi'})(\delta\wh{\psi})=(\delta,\delta)_c(\wh{\psi'},\wh{\psi})_{c-1}\neq 0\text{ if and only if $\wh{\psi}=\wh{\psi'}.$}\] Thus \[\Phi_c(\wh{\psi'})(\delta\wh{\psi})\neq 0\text{ if and only if $\wh{\psi}=\wh{\psi'}$,}\] in which case, \[(\wh{P}(\delta \wh{\psi}), \wh{h'}\wh{\psi'})_c=(\wh{P}\Phi_c(\wh{\psi})(\delta \wh{\psi}), \wh{h'})_c=(\wh{\psi},\wh{\psi})_{c-1}(\wh{P}(\delta),\wh{h'})_c=(\wh{\psi},\wh{\psi})_{c-1}(\wh{h},\wh{h'})_c\neq 0\text{ iff $\wh{h}=\wh{h'}$}.\]

As a consequence of (a)(b)(c), $\{\ol{\Phi}_c(\wh{P}(\delta\wh{\psi}))\}$ forms a new basis of 
$\ol{\Phi}_c(\mathcal{S}_{j,\perp_c})$ that is contained in $F^{\ul{\ind}}_j \Lrm_c$. Hence the lemma follows.
\end{proof}
Consider another important subspace of $(\afrak^{j+1})_{j}^{\perp_c}$:
 \begin{align}
     \label{S}
\mathcal{S}_j:=\langle\ul{h\psi}\in  (\afrak^{j+1})_{j}^{\perp_c}|h\in \Hscr_c, \psi\in I^W_{c-1}\rangle.
 \end{align}
\begin{lem}\label{Icperp}
$\mathcal{S}_{j,\perp_c}=(\mathcal{S}_j)^{\perp_c}\cap (\afrak^{j+1})_{j}^{\perp_c}$.
\end{lem}
\begin{proof}
By Lemma \ref{basis}, $\mathcal{S}_{j,\perp_c}$ and $\mathcal{S}_j$ together span $(\afrak^{j+1})_{j}^{\perp_c}$. Therefore we only need to show that $\mathcal{S}_{j,\perp_c}\subset (\mathcal{S}_j)^{\perp_c}\cap (\afrak^{j+1})_{j}^{\perp_c}$. To do so, by the proof of Lemma \ref{notinI}, it suffices to show that $(\wh{P}(\delta \wh{\psi}),\ul{h \phi})_c=0$ for all $\ul{h\phi}\in (\afrak^{j+1})_{j}^{\perp_c}$ such that $h\in \Hscr_c$ and  $\phi\in I^W_{c-1}$ for $\wh{P},\wh{\psi}$ as in the proof of Lemma \ref{notinI}.

Because of Lemma \ref{isotypic}, we have that $(\wh{P}(\delta \wh{\psi}),h \phi)_c=(\delta \wh{\psi}, \Phi_c(\wh{P})h\phi)_c$
is nonzero if and only if there exists some $\phi^\circ\in \C[\hfrak]^W$ such that \[(\delta\wh{\psi},\delta \phi^\circ)_c=(\delta\wh{\psi},\Phi_c(\wh{P})h\phi)_c.\] By Lemma \ref{isotypic}, $\delta \phi^\circ\in I_{c-1}^{\sign}$. Moreover, note that $I^{\sign}_{c-1}=\delta I_{c-1}^W$ and hence $\phi^\circ\in I_{c-1}$. We conclude the lemma by 
\[(\delta \wh{\psi}, \Phi_c(\wh{P})h\phi)_c=(\delta,\delta)_c(\wh{\psi}, \phi^\circ)_{c-1}=0.\qedhere\]
\end{proof}

\subsection{Criteria for induction}\label{criteria}
Showing $F^{\ul{\ind}}=F^\afrak$ on $\mathcal{S}_{j}$ is much more involving and will occupy the rest of the paper. With doing induction on indexes of the filtrations in mind, we prove Proposition \ref{yphi}.
\begin{lem}\label{eLc-1}
   The composition \[F^\ord_0( \Hrm_c\e_-)\otimes_\C \e\Lrm_{c-1}\to \Hrm_c  \e_-\otimes_{\e\Hrm_{c-1}\e}\e\Lrm_{c-1}\cong \Lrm_c\to \Lrm_c/(I_{c-1}^W)\] is an isomorphism.
\end{lem}
\begin{proof}
  Note that  $\big(\Lrm_c/(I_{c-1}^W)\big)^W=\e\Lrm_c/I_{c-1}^W=\e\Lrm_{c-1}$ and so by Lemma \ref{basis}, \[\Lrm_c/(I_{c-1}^W)=\Hscr_c\cdot \big(\Lrm_c/(I_{c-1}^W)\big)^W.\] 
   Moreover, $F^\ord_0( \Hrm_c\e_-)\cong \C[x_1,\cdots, x_n]$.  It remains to notice that \[\C[x_1,\cdots, x_n]\cdot \e\Lrm_{c-1}=(\C[x_1,\cdots, x_n]/\afrak) \cdot \e\Lrm_{c-1}=\Hscr_c\cdot \Lrm_c^W/I_{c-1}^W.\qedhere\]
\end{proof}
\begin{prop}\label{yphi}
For $j\geq 1$, suppose $F^{\ul{\ind}}_{j-1}\Lrm_c=F^\afrak_{j-1}\Lrm_c$. The following statements are equivalent.
\begin{enumerate}[(i)]
\item $F_j^{{\ind}}\ol{\Phi}_c(\mathcal{S}_j)= F_j^\afrak\ol{\Phi}_c(\mathcal{S}_j)$.
 \item For any nonzero $\psi\in\mathcal{S}_j$, if $(y_k-y_{k+1})\psi\in \afrak^{j}$ for all $k\in[1,n-1]$, then $\psi=0$ in $\Lrm_c$.
\end{enumerate}
\end{prop}
\begin{proof}
We first prove an important identity  (cf. (\ref{euler})). Take $ {\psi}\in\Scal_j$. 
    We may assume that $ {\psi}\in \e_\sigma\Lrm_c\subset \Lrm_c$ for some $\sigma\in\hat{W}$: Otherwise suppose $ {\psi}=\sum_{\sigma\in \hat{W}} {\psi}_\sigma$ for $ {\psi}_\sigma\in \e_\sigma\Lrm_c$. By Corollary \ref{coro_isotypic}, $ {\psi}_\sigma= {\sum h\phi}$ for $h\in \Hscr_c^\sigma$ and $\phi\in\C[\hfrak]^W$. In particular, there is no possible cancellation between $ {\psi}_\sigma$ and $ {\psi}_{\sigma'}$ for $\sigma\neq \sigma'$. By the definition of $\Scal_j$, we must have $ {\psi}_\sigma\in \Scal_j$ as well.

From $\mathbf{h}  {\psi}=\nu {\psi}$ and (\ref{formula_h}) we deduce 
\[\nu {\psi}=\bigg(\frac{1}{n}\sum_{i<j} (x_i-x_j)(y_i-y_j)+\frac{n-1}{2}-c\sum_{s\in S}s\bigg)  {\psi}.\]
Since $\sum_{s\in S}s$ lies in the center of $\C[W]$, it acts on ${\psi}\in \e_\sigma\Lrm_c$ by an integer \cite[4.1]{beg2}:
\begin{align}
    \label{formula_xi}
\xi=|S|\frac{\chi(s,\sigma)}{\chi(1,\sigma)}.\end{align} 
Here $s$ is a fixed reflection and $\chi$ indicates character. By Lemma \ref{ef} (5), we only need to consider highest or lowest weight vectors under the action of $\{\ebf,\fbf,\hbf\}$. Therefore, we may make the following assumptions: when $\frac{n-1}{2}-c\xi>0$, we assume $\nu\leq 0$; when $\frac{n-1}{2}-c\xi<0$, we assume $\nu\geq 0$. In either case, we end up with the identity:
\begin{align}
    \label{euler}
 {\psi}=\frac{1}{n(\nu+c\xi-(n-1)/2)}\sum (x_k-x_\ell)(y_k-y_\ell) {\psi}.\end{align}
Suppose $\frac{n-1}{2}-c\xi=\nu=0$, or equivalently,
\[m=\frac{\chi(1,\sigma)}{\chi(s,\sigma)}.\]
However, by (\ref{formula_xi}), $\frac{\chi(1,\sigma)}{\chi(s,\sigma)}$ divides $|S|=\frac{n(n-1)}{2}$. Since $m>n$ by assumption, $\frac{\chi(1,\sigma)}{\chi(s,\sigma)}$ cannot be coprime to $n$. A contradiction. Therefore the situation of $\frac{n-1}{2}-c\xi=\nu=0$ never occurs and $ {\psi}$ has to satisfy (\ref{euler}). 

Suppose \[\frac{1}{n(\nu+c\xi-(n-1)/2)}(y_k-y_\ell){\psi}=\alpha_{k\ell}+\beta_{k\ell},\quad \alpha_{k\ell}\in \afrak^j,\beta_{k\ell}\in(\afrak^j)^{\perp_c}.\]
   By (\ref{euler}), $\psi=\sum_{k,\ell} (x_k-x_\ell)(\alpha_{k\ell}+\beta_{k\ell})$.

    $\mathbf{(i) \Rightarrow (ii)}$: By assumption, $\ol{\Phi}_c(\psi)\in F_j^\afrak\Lrm_c=F_j^{\ul{\ind}}\Lrm_c$ with $j>0$. 
    Suppose that $(y_k-y_{k+1})\phi\in \afrak^{j}$ for all $k\in[1,n-1]$. Then \begin{align}
        \label{formula_beta=0}
    \psi=\sum_{k,\ell}(x_k-x_\ell)\alpha_{k\ell}.
    \end{align}
    
    By the definition of $F^{\ul{\ind}}$, $\ol{\Psi}_c(\psi)= \sum_{\alpha+\beta=j} \psi_\alpha \otimes \psi_\beta$ where $\psi_\alpha\in F^{\ord}_\alpha \Hrm_c \e_-$ and $\psi_\beta\in F^{\ul{\ind}}_\beta \e\Lrm_{c-1}$. Since the image of $\psi$ in $\Lrm_c/(I_{c-1}^W)$ is $0$, given Lemma \ref{eLc-1}, we may assume $\alpha>0$. As a result,
        $\ol{\Phi}_c(\psi)$ can be written in the form of $\sum_k (y_k-y_{k+1})\ol{\Phi}_c(\psi_k)$ for some $\ol{\Phi}_c(\psi_k)\in F_{j-1}^{\ul{\ind}}$ which equals $F_{j-1}^{\afrak}$ by our assumption. Therefore  $\ol{\Phi}_c(\psi)\in \ol{\Phi}_c(\sum_{k=1}^n (x_k-x_{k+1}) (\afrak^j)^{\perp_c})$, or equivalently, $\psi \in \sum_{k=1}^n (x_k-x_{k+1}) (\afrak^j)^{\perp_c}$. This together with (\ref{formula_beta=0}) forces $\psi$ to be $0$.

   $\mathbf{(ii)\Rightarrow (i)}$: Take $ {\psi}\in\Scal_j$. Then $\ol{\Phi}_c(\psi)\in F_j^\afrak\Lrm_c$.
   The statement (ii) is equivalent to an isomorphism
    \[(\afrak^{j+1})^{\perp_c}\cong \sum_{k=1}^n x_k (\afrak^j)^{\perp_c}.\] Therefore, \[\ol{\Phi}_c(\psi)\in \sum_k (y_k-y_{k+1})\ol{\Phi}_c((\afrak^j)^{\perp_c})\subset \sum_k (y_k-y_{k+1})F_{j-1}^{\ul{\ind}}\Lrm_c\subset F_{j}^{\ul{\ind}}\Lrm_c.\] Thus (ii)$\Rightarrow$ (i) is proved.
\end{proof}

\section{The kernel of the Dunkl form and the coinvariant algebra}\label{sectionkernel}
   Our goal in this and the next sections is to show that (ii) in Proposition \ref{yphi} holds. To study elements in $\Scal_j$, we recollect some facts about the kernel $I_c$.
\subsection{The kernel of the Dunkl form}\label{kernel}
Recall that $c=\frac{m}{n}>1$ for positive integer $m$ coprime to $n$. Define $u_i$ to be the elementary symmetric polynomials satisfying $\prod_{i=1}^n (1-x_iz)=\sum u_i z^i$ (put $u_1=0$). Let $v^{(c)}_i$ be the polynomials such that the formal Taylor expansion of the following equation holds:
\[\left(\sum u_i z^i\right)^c=\sum v^{(c)}_iz^i.\]
 For $i=1,\dots, n$, define  \begin{align*}
     f_i&=\text{Coef}_{z^m}\left((1-x_iz)^{-1}\prod_{k=1}^n (1-x_kz)^c\right)=\sum_{j=0}^m x_i^{j}v^{(c)}_{m-j}\\
     &=\text{Coef}_{z^m}\left(\prod_{l\neq i}(1-x_lz)\prod_{k=1}^n (1-x_kz)^{c-1}\right)=\sum_{j=0}^{n-1} \left((-1)^j\sum_{l_1<\cdots<l_j, l_1,\cdots, l_j\neq i}  x_{l_1}\cdots x_{l_j}\right)v^{(c-1)}_{m-j }.
     \end{align*}
\begin{lem}\label{lastlemma}
For $k<n$,  one has that \[\displaystyle x_i^k+x_i^{k-2}u_2+\cdots +u_k=(-1)^k\sum_{j_1<\cdots<j_k,j_\ell\neq i} x_{j_1}\cdots x_{j_k}.\] As a consequence, 
    $\displaystyle x_i^k\equiv(-1)^k\sum_{j_1<\cdots<j_k,j_\ell\neq i} x_{j_1}\cdots x_{j_k}$  $\mathrm{mod\ }\afrak$.
\end{lem}
\begin{proof}
    We do induction on $k$. When $k=1$, it follows from $x_1+x_2+\cdots +x_n\in \afrak$. When $k>1$, assume the lemma is proved for integers less than $k$. Then by the induction hypothesis
    \begin{align*}
    x_i^k=&x_i\left((-1)^{k-1} \sum_{j_1<\cdots<j_{k-1},j_\ell\neq i} x_{j_1}\cdots x_{j_{k-1}}-(x_i^{k-3}u_2+\cdots +u_{k-1})\right)\\
    =&(-1)^{k-1} \sum_{j_1<\cdots<j_{k-1},j_\ell\neq i} x_ix_{j_1}\cdots x_{j_{k-1}}-(x_i^{k-2}u_2+\cdots +x_iu_{k-1})\\
    =&(-1)^{k-1}\left(u_k -\sum_{j_1<\cdots<j_{k},j_\ell\neq i} x_{j_1}\cdots x_{j_{k}}\right)-(x_i^{k-2}u_2+\cdots +x_iu_{k-1})\\
    =&(-1)^k\sum_{j_1<\cdots<j_k,j_\ell\neq i} x_{j_1}\cdots x_{j_k}-(x_i^{k-2}u_2+\cdots +u_k).\qedhere
    \end{align*}
\end{proof}
As a consequence,
\begin{align}
    \label{generalfi}
f_i= \sum_{j =0}^{n-1}  ( x_i^j+x_i^{j-2} u_2+\cdots + u_j )v^{(c-1)}_{m-j}. \end{align}
According to \cite[(4.2)]{gorskyarc}, the summand $v^{(c-1)}_m$ in $f_i$ can be expressed by
\[v^{(c-1)}_m=(2(c-1)-(m-2))v_{m-2}^{(c-1)}u_2+\cdots +((n-1)(c-1)-(m-n-1))v^{(c-1)}_{m-n+1}u_{n-1}.\]
Suppose $kn<m<(k+1)n$. Then $v^{(c-1)}_{m-l }\in \afrak^k$ for $l =m-nk,\cdots, n-1$, and $v^{(c-1)}_{m-l }\in \afrak^{k+1}$ when $l =0,\cdots, m-nk-1$. Hence
\begin{align}
    \label{fmodulo}
f_i\equiv \sum_{j =m-nk}^{n-1} x_i^j v^{(c-1)}_{m-j}\quad \text{mod }\afrak^{k +1}.
\end{align}
\begin{thm}\cite[Chapter 5]{dunklintertwining}\cite[Proposition 3.1] {chmutovaetingof}\cite[Theorem 4.3]{gorskyarc} We have that
\begin{enumerate}[(i)]\label{arc}
    \item $\displaystyle I_c=\sum_i \C[\hfrak]f_i$  and $\sum_i \C f_i $ form the standard permutation representation of $S_n$.
    \item $\displaystyle \sum_i x_i^\ell f_i=-\frac{m+\ell}{c}v^{(c)}_{m+\ell}-\sum_{j=1}^{\ell-1} v^{(c)}_{m+j}p_{\ell-j}$.
    \item  $\displaystyle I_c^W=(v^{(c)}_{m+1},\cdots, v^{(c)}_{m+n-1})$ and $v^{(c)}_{k}\in I_c^W$ for $k>m+n-1$.
\end{enumerate}
\end{thm}
Given the expression (\ref{generalfi}) and (iii) of Theorem \ref{arc}, we see that $I_c$ is contained in the ideal generated by $I_{c-1}^W$.
\if force As $(\sum u_i z^i)^c=(\sum u_i z^i)^{c-1}((\sum u_i z^i)$, we see
\[    v_j^{(c)}=\sum_{k=\text{max}\{0,j-n\}}^jv_k^{(c-1)}u_{j-k}\]
\[    v_j^{(c-1)}=\sum_{k+j_1+\cdots +j_\gamma=j}(-1)^\gamma v_k^{(c)}u_{j_1}u_{j_2}\cdots u_{j_\gamma}\]
Therefore 
by Theorem \ref{arc}, we see that $I_c^W\subset I_{c-1}^W$.\fi 

\begin{rem}
The generators $v_{m-n+1}^{(c-1)}, \cdots, v_{m-1}^{(c-1)}$ of $I_{c-1}^W$ form a regular sequence. As a result,  $\dim (\Lrm_c/(I_{c-1}^W))=(m-n+1)\cdots(m-1)=\dim \Hscr_c\cdot \dim(\e\Lrm_{c-1})$, which implies that $\Lrm_c/(I_{c-1}^W)\cong \Hscr_c\otimes_\C \e\Lrm_{c-1}$.
\end{rem}

\subsection{Equality on $F_1$}\label{transform}
In this and the next subsections, we will show that $F_1^{\ul{\ind}}=F_1^{\afrak}$. In the process we will acquire all the ingredients needed for proving $F^{\ul{\ind}}=F^{\afrak}$ in general.

Take a nonzero polynomial $\phi\in (\afrak^2)^{\perp_c}\cap \afrak$. Assume $\phi\equiv \sum_k h_kp_k$ mod $\afrak^2$ for $h_k\in \Hscr_c$. By Lemma \ref{notinI} and \ref{Icperp}, we may assume, in addition, $p_k\in {I_{c-1}^W}$ mod $\afrak^2$. From Theorem \ref{arc}, we know that ${I_{c-1}^W}$ mod $\afrak^2$ is generated by ${v^{(c)}_{m-n+1}},\cdots, {v^{(c)}_n}$. Therefore to have ${I_{c-1}^W}\neq 0$ mod $\afrak^2$, the integer $m$ has to satisfy $n<m<2n$,

\textbf{$\spadesuit$ Hence we will assume $n<m<2n$ for Sections \ref{subsection_matrices} and \ref{linearalgebra}.}

Consider the orthogonal lift $\ul{\phi}\in \Scal_1$ where $\phi:=\sum_{k=m-n+1}^n h_kp_k$. Given Proposition \ref{yphi} (ii)$\Rightarrow$(i), to show that $F_1^{\ul{\ind}}\Scal_1=F_1^\afrak S_1$, it suffices to show that
\begin{align}
    \label{statement_on_F1}
\textbf{if $(y_i-y_{i+1})\phi\in \afrak$ for all $i$, then $\ul{\phi}\in I_c$.}\end{align}
\subsection{Interpreting (\ref{statement_on_F1}) via two matrices}\label{subsection_matrices}
We know that any element in $I_c$ has the form \begin{align}
    \label{fandxk}
\sum_{i=1}^n \phi_i f_i= \sum_{i=1}^n \phi_i \sum_{j=0}^m x_i^{m-j}v^{(c)}_{j}.\end{align}
 
Replacing $f_i$ for $i=1,\dots, n$ by 
\[f_i\equiv x_i^{n-1}v^{(c-1)}_{m-n+1}+\cdots +x_i^{m- n}v^{(c-1)}_{ n} \text{ mod } \afrak^{2}\] 
and $v_j$ for $j=m-n+1,\dots, n$ by
\[v^{(c)}_j\equiv c u_j\quad \text{mod}\ \afrak^2\]
 we have the further expansion \begin{align}
    \label{linearcombination}
\sum_{i=1}^n \phi_i f_i\equiv (c-1)\sum_{i=1}^n \phi_i \sum_{j=m-n+1}^{ n}x_i^{m-j}u_j\text{ mod } \afrak^2. \end{align}

Let $\Rcal:=\C[\hfrak]/\afrak$ be the coinvariant algebra.  Using row vector multiplication, we have two linear maps
        \begin{align}
            \label{sequence}
    \Rcal^n\xrightarrow{A_{m-n,n-1}}\Rcal^{2n-m}\xrightarrow{B_{m-n,n-1}} \Rcal^n 
        \end{align} defined by the following matrices\begin{align}
      \label{AB}
  A_{m-n,n-1}:=\begin{pmatrix}
            x_1^{n-1}&x_1^{n-2}&\cdots&x_1^{m-n}
            \\
            x_2^{n-1}&x_2^{n-2}&\cdots&x_2^{m-n}\\
            &&\cdots\\
            x_n^{n-1}&x_n^{n-2}&\cdots&x_n^{m-n}
        \end{pmatrix},\quad B_{m-n,n-1}:=\begin{pmatrix}
            x_1^{m-n}&x_2^{m-n}&\cdots&x_n^{m-n}\\
            x_1^{m-n+1}&x_2^{m-n+1}&\cdots&x_n^{m-n+1}\\
            &&\cdots\\
            x_1^{n-1}&x_2^{n-1}&\cdots&x_n^{n-1}
        \end{pmatrix}.\end{align}
        
        We will simply write $A:=A_{m-n,n-1}$, $B:=B_{m-n,n-1}$ when there is no ambiguity.
        
  \begin{lem}\label{ABequivalent}
  \begin{enumerate}[(a)]
      \item $\ul{\sum_{m-n+1}^n h_kp_k}\in I_c$ holds if and only if $((m-n+1)h_{m-n+1},\cdots, nh_n)\in  \im(A)$.
      \item $(y_i-y_{i+1})\sum_{k=m-n+1}^n h_kp_k\in \afrak$ holds for all $i$ if and only if $((m-n+1)h_{m-n+1},\cdots, nh_n)\in  \Ker(B)$. 
  \end{enumerate}
       \end{lem}
       \begin{proof}
       Using (\ref{linearcombination}), we see that
        \[\sum \phi_i f_i\equiv (c-1)\sum_{i=1}^n \phi_i \sum_kx_i^{m-k}ku_k\equiv (c-1)\sum_k (\sum_{i=1}^n\phi_i x_i^{m-k})u_k \text{ mod $\afrak^2.$}\]
          By Newton's identity  $p_k\equiv ku_k$ mod $\afrak^2$, for $k=2,\cdots, n$.
        Therefore, the condition $\ul{\sum_{m-n+1}^n h_kp_k}\in I_c$ implies the following equations 
\[\sum kh_ku_k\equiv (c-1)\sum_k (\sum_{i=1}^n\phi_i x_i^{m-k})u_k \text{ mod $\afrak^2$}\]
which implies:  \[((m-n+1)h_{m-n+1},\cdots, nh_n)=(c-1)(\phi_1 \cdots \phi_n) A.\] This proves (a).

Because $p_k$ is symmetric, by \cite[Corollary 5.3]{gorskyarc} we have that 
\[(y_{i}-y_{i+1})(h_kp_k)=((y_{i}-y_{i+1}) h_k)p_k+h_k((y_{i}-y_{i+1})p_k).\]
The element $(y_{i}-y_{i+1})\sum_k h_ky_k$ is in $\afrak$ if and only if $\sum_k h_k((y_{i}-y_{i+1})p_k)$ is in $\afrak$. The latter element is equal to $\sum_k kh_k (x_i^{k-1}-x_{i+1}^{k-1})$, because $y_i-y_{i-1}$ acts on symmetric polynomials by $\frac{\partial}{\partial_{x_i}}-\frac{\partial}{\partial_{x_{i+1}}}$.

 This can be translated to 
       \[((m-n+1)h_{m-n+1},\cdots, nh_n)\begin{pmatrix}
           x_1^{m-n}-x_2^{m-n}&x_2^{m-n}-x_3^{m-n}&\cdots&x_{n-1}^{m-n}-x_n^{m-n}\\
            x_1^{m-n+1}-x_2^{m-n+1}&x_2^{m-n+1}-x_3^{m-n+1}&\cdots&x_{n-1}^{m-n+1}-x_n^{m-n+1}\\
            &&\cdots\\
            x_1^{n-1}-x_2^{n-1}&x_2^{n-1}-x_3^{n-1}&\cdots&x_{n-1}^{n-1}-x_n^{n-1}
       \end{pmatrix}\equiv 0\text{ mod }\afrak.\]
       
       We add in a new column to the matrix $B'$ in the last equation to define \[B''=\begin{pmatrix}
           x_1^{m-n}-x_2^{m-n}&\cdots&x_{n-1}^{m-n}-x_n^{m-n}&p_{m-n}\\
            x_1^{m-n+1}-x_2^{m-n+1}&\cdots&x_{n-1}^{m-n+1}-x_n^{m-n+1}&p_{m-n+1}\\
            &&\cdots\\
            x_1^{n-1}-x_2^{n-1}&\cdots&x_{n-1}^{n-1}-x_n^{n-1}&p_{n-1}
       \end{pmatrix}.\]
       
       As all entries in the last column of $B''$ lie in $\afrak$ and that $B$ can obtained from $B''$ by elementary column  operations, we conclude that $\Ker(B)=\Ker(B'')=\Ker(B')$ as desired.
       \end{proof}
Therefore our goal becomes to show that $\im(A)= \Ker(B)$. We will first show $\im(A)\subset \Ker(B)$ and then show that $\rank(A)=\rank(B)$. 

\subsection{The containment ``$\im(A)\subset \Ker(B)$"}\label{linearalgebra}

        Composing $A$ and $B$ gives 
        \[  AB=\begin{pmatrix}
            nx_1^{m-1}&\sum_{j=m-n}^{n-1}x_1^jx_2^{m-j-1}&\cdots &\sum_{j=m-n}^{n-1} x_1^jx_n^{m-j-1}\\
            \sum_{j=m-n}^{n-1} x_2^jx_1^{m-j-1}&nx_2^m&\cdots &\sum_{j=m-n}^{n-1} x_2^jx_n^{m-j-1}\\
            &&\cdots\\
            &&\cdots\\
            \sum_{j=m-n}^{n-1} x_1^jx_n^{m-j-1}&&\cdots &nx_n^m
        \end{pmatrix}.\]
We will show $AB=0$ below, which implies that $\im(A)\subset \Ker(B)$ for $n<m<2n$. 
\begin{lem}\label{xixj}
  For $i\neq j$ and $k<n-1$, we have \[x_i^k+x_i^{k-1}x_j+\cdots +x_ix_j^{k-1}+x_j^k\equiv(-1)^k \sum_{j_1<\cdots<j_k,j_\ell\neq i,j} x_{j_1}\cdots x_{j_k}  \quad \mathrm{mod\ }\afrak.\]
\end{lem}
\begin{proof}
We do induction on $k$. When $k=1$, this follows from $x_1+x_2+\cdots +x_n\in \afrak$. Assume the lemma is proved for integer less than $k$ where $k\geq 2$. 
By Lemma \ref{lastlemma}, we have
\begin{align}
    \label{1x}
x_j^k=x_j((-1)^{k-1}\sum_{j_1<\cdots<j_{k-1},j_\ell\neq j} x_{j_1}\cdots x_{j_{k-1}})\text{ mod }\afrak.\end{align}
On the other hand, notice that
\begin{align}
    \label{kx}-\sum_{j_1<\cdots<j_{k},j_\ell\neq i,j} x_{j_1}\cdots x_{j_{k}}\equiv x_i\sum_{j_1<\cdots<j_{k-1},j_\ell\neq i} x_{j_1}\cdots x_{j_{k-1}}+x_j\sum_{j_1<\cdots<j_{k-1},j_\ell\neq j} x_{j_1}\cdots x_{j_{k-1}}\text{ mod }\afrak.\end{align}
Now by the induction hypothesis (modulo $\afrak$ below)
\begin{align*}
    x_i^k\equiv &x_i\left( (-1)^{k-1}\sum_{j_1<\cdots<j_{k-1},j_\ell\neq i,j} x_{j_1}\cdots x_{j_{k-1}} - (x_i^{k-2}x_j+\cdots +x_j^{k-1})\right) \\
    \equiv &(-1)^{k-1}\sum_{j_1<\cdots<j_{k-1},j_\ell\neq i,j} x_ix_{j_1}\cdots x_{j_{k-1}}-(x_k^{k-1}x_j+\cdots +x_ix_j^{k-1})\\
    (\text{by }(\ref{kx}))\equiv &(-1)^k\left(\sum_{j_1<\cdots<j_{k},j_\ell\neq i,j} x_{j_1}\cdots x_{j_{k}}+(\sum_{j_1<\cdots<j_{k-1},j_\ell\neq j} x_jx_{j_1}\cdots x_{j_{k-1}})\right)-(x_k^{k-1}x_j+\cdots +x_ix_j^{k-1})\\
    (\text{by }(\ref{1x}))\equiv &(-1)^k\sum_{j_1<\cdots<j_{k},j_\ell\neq i,j} x_{j_1}\cdots x_{j_{k}}-(x_k^{k-1}x_j+\cdots +x_ix_j^{k-1}+x_j^k).\qedhere
\end{align*}
\end{proof}
\begin{coro}\label{AB=0}For $m\geq n$, 
    the element $\phi_{ij}:=x_i^{m-n}x_j^{n-1}+x_i^{m-n+1}x_j^{n-2}+\cdots +x_i^{n-1}x_j^{m-n}$ is in $\afrak$. 

Thus, $AB=0$, $\im(A)\subset \Ker(B)$ and $\rank(A)+\rank(B)\leq (2n-m)n!$.
\end{coro}
\begin{lem}\label{less}
    $\rank(A)=\rank(B)$ and so $\rank(A)\leq \frac{2n-m}{2}n!$.
\end{lem}
\begin{proof}
Computing the rank of $A$ by taking the span of its column vectors we obtain\[\text{rank}(A)=\dim_\C\{\phi_1(x_1^{n-1},\cdots, x_n^{n-1})+\cdots+\phi_{2n-m}(x_1^{m-n},\cdots, x_n^{m-n}),\phi_i \in \Rcal\}.\] Computing the rank of $B$ by taking the span of its row vectors we obtain
\[\text{rank}(B)=\dim_\C\{\psi_1(x_1^{m-n},\cdots, x_n^{m-n})+\cdots+\psi_{2n-m}(x_1^{n-1},\cdots, x_n^{n-1}),\psi_i \in \Rcal\}.\]
As a result, rank$(A)=$rank$(B)$.
\end{proof}
\subsection{Computing the rank}Given Corollary \ref{AB=0}, to show that $\im(A)=\Ker(B)$, it suffices to show that rank$(A)+$rank$(B)=(2n-m)n!$, which, given Lemma \ref{less}, amounts to showing that $\rank(A)\geq \frac{2n-m}{2}n!$. 

Write \[V_i:=\Rcal(x_i^{n-1},\cdots x_i^{m-n}),\] the $\Rcal$-submodule in $\Rcal^{2n-m}$ generated by the $i$-th row vector of $A$.  Then 
\[\rank(A)=\rank(B)=\dim_\C (V_1+\cdots +V_n).\]

From now on, let $\Rcal_i$ denote the coinvariant algebra of $i$ variables and \[\im_n(x_i^{n-1}):=\im(x_i^{n-1}, \Rcal_n\to\Rcal_n).\] Note that the composition \[\im_n(x_i^{n-1})\xrightarrow{(x_i^{n-1})^T} \Rcal_n\to \Rcal_n/\im_n(x_i)\] is an isomorphism and $\Rcal_n/\im_n(x_i)\cong \Rcal_{n-1}$. Let $\iota_i: \im_n(x_i^{n-1})\to \Rcal_{n-1}$ denote this isomorphism. Then \[\iota_1(\im_n(x_1^{n-1})\cap \im_n(x_2^{n-1}))=\Ker_{n-1}(x_2)=\im_{n-1}(x_2^{n-2})\cong \Rcal_{n-2}.\]

When $m=2n-1$, 
\[V_i=\Ker (x_i)=\im(x_i^{n-1}).\] From the discussion above, we further deduce that $\dim(V_i)=(n-1)!$ and $\dim(V_i\cap V_j)=(n-2)!$ when $1\leq i,j\leq n$ and  $i\neq j$. In general:
\begin{lem}\label{dimcount}
\begin{itemize}
    \item 
    When $n<m<2n$, $\dim (V_i)=(2n-m)(n-1)!$
    \item When $1\leq i,j\leq n$ and  $i\neq j$, $\dim (V_i\cap V_j)=(2n-m)(n-2)!$ 
    \end{itemize}
\end{lem}
\begin{proof}
    Define an operator $J_i: \Rcal_n^{2n-m}\to \Rcal_n^{2n-m}$ given by the following matrix, using row multiplication.
\[J_i:=\begin{pmatrix}
    1&0&\cdots&0&0\\
    -x_i&1&\cdots&0&0\\
    0&-x_i&\cdots&0&0\\
    &&\cdots&\\
    0&0&\cdots &1&0\\
    0&0&\cdots&-x_i&x_i^{2n-m}
    \end{pmatrix}\]
For example $J_i=x_i$ when $m=2n-1$. The vector space
\[V_i=\Ker(J_i)\cong \Ker(x_i^{2n-m})\] is of dimension $(2n-m)(n-1)!$. 
Let $\begin{pmatrix}
    J_i&J_j.
\end{pmatrix}$ denote concatenation of the matrices $J_i$ and $J_j$. When $i\neq j$, \[\dim(V_i\cap V_j)=\dim \Ker\begin{pmatrix}
    J_i&J_j
\end{pmatrix}.\]
Using elementary matrix transformations, $\begin{pmatrix}
    J_i&J_j
\end{pmatrix}$ can be transformed into
\[\begin{pmatrix}
    1&0&\cdots &0&0&    0&0&\cdots&0&0\\
    0&1&\cdots &0&0&    0&0&\cdots&0&0\\
    &&\cdots&&&      &&\cdots\\
    0&0&\cdots &1&0&    0&0&\cdots&0&0\\
    0&0&\cdots &0 &x_i^{2n-m}&   0&0&\cdots&x_i-x_j&x_j^{2n-m} 
\end{pmatrix}\]
Therefore, $\dim \Ker\begin{pmatrix}
    J_i&J_j
\end{pmatrix}=\dim \Ker(x_i^{2n-m}\ \ x_i-x_j\ \ \ x_j^{2n-m})$.

Since $(x_i-x_j)|(x_i^{2n-m}-x_j^{2n-m})$, we get \[\Ker(x_i^{2n-m}\ \ x_i-x_j\ \ x_j^{2n-m})=\Ker(x_i^{2n-m}\ \ x_i-x_j).\]

Without loss of generality, we may assume $i=1$ and $j=2$. 

By Lemma \ref{ker(x_1-x_2)} below, one has $\Ker(x_1-x_2)=\sum_{j=1}^{n-1} \psi_j\Rcal_{n-2}$ where \[\psi_j:=x_1^jx_2^{n-2}+2x_1^{j+1}x_2^{n-3}+\cdots +(n-j)x_1^{n-1}x_2^{j-1}\in \Ker(x_1-x_2).\]

In particular, \[\psi_{m-n},\psi_{m-n+1},\cdots, \psi_{n-1}\in \Ker(x_1^{2n-m}).\] Since multiplication by $x_1^{2n-m}$ preserves the basis $x_1^{a_1}x_2^{a_2}\cdots x_{n-1}^{a_{n-1}}$ ($0\leq a_i\leq n-i$) (possibly sending some elements to $0$), we conclude that \[\psi_{m-n}\Rcal_{n-2}+\cdots+\psi_{n-1}\Rcal_{n-2}= \Ker(x_1^{2n-m})\cap \Ker(x_1-x_2),\]
and so $\dim (V_1\cap V_2)=\dim \Ker(J_1\ \ J_2)=\dim (\Ker(x_1^{2n-m})\cap \Ker(x_1-x_2))= (2n-m)(n-2)!$.
\end{proof}
\begin{lem}\label{ker(x_1-x_2)}
Let $\Rcal_{n-2}\subset \Rcal_n$ be the subspace generated by $x_3^{a_3}x_4^{a_4}\cdots x_{n-1}^{a_{n-1}}$, $0\leq a_i\leq n-i$. For $1\leq j\leq n-1$, write \[\psi_j:=x_1^jx_2^{n-2}+2x_1^{j+1}x_2^{n-3}+\cdots +(n-j)x_1^{n-1}x_2^{j-1}.\]
Then $\sum_{j=1}^{n-1} \psi_j\Rcal_{n-2}= \Ker(x_1-x_2,\Rcal_n\to \Rcal_n)$. 
\end{lem}
\begin{proof}
We compute 
\begin{align*}
    &(x_1-x_2)(x_1^jx_2^{n-2}+2x_1^{j+1}x_2^{n-3}+\cdots +(n-j)x_1^{n-1}x_2^{j-1})\\
    =&-x_1^jx_2^{n-1}+x_1^{j+1}x_2^{n-2}-2x_1^{j+1}x_2^{n-2}+2x_1^{j+2}x_2^{n-3}+\cdots -(n-j)x_1^{n-1}x_2^j+(n-j)x_1^nx_2^{j-1}\\
    =&-x_1^jx_2^{n-1}-x_1^{j+1}x_2^{n-2}-\cdots -x_1^{n-1}x_2^j.
    \end{align*}
    The element in the last line belongs to $\afrak$ by Corollary \ref{AB=0}.

Next, consider the basis $x_1^{a_1}\cdots x_{n-1}^{a_{n-1}}$ ordered by lexicographic order on the powers. With respect to this basis, multiplication by $x_1$ can be expressed by \[N\otimes I_{(n-1)!}\] where $I_k$ stands for the $k\times k$ identity matrix and 
\[N:=\begin{pmatrix}
    0&0&\cdots &0&0\\
    1&0&\cdots &0&0\\
    0&1&\cdots &0&0\\
    &&\cdots\\
    0&0&\cdots &1&0
\end{pmatrix}_{n\times n}\]
Multiplication by $x_2$ can be expressed by 
\[\begin{pmatrix}
    0&0&\cdots&0&-N^{n-1}\\
    I&0&\cdots &0&-N^{n-2}\\
    0&I&\cdots &0&-N^{n-3}\\
    &&\cdots &&\\
    0&0&\cdots&I&-N
\end{pmatrix}
\otimes I_{(n-2)!}\]
Observe that, using elementary row operations, the following matrix represents multiplication by $x_2-x_1$ \[\begin{pmatrix}
    -N&0&\cdots&0&-N^{n-1}\\
    I&-N&\cdots &0&-N^{n-2}\\
    0&I&\cdots &0&-N^{n-3}\\
    &&\cdots &&\\
    0&0&\cdots&I&-2N
\end{pmatrix}
\otimes I_{(n-2)!}\] which can be transformed to 
\[\begin{pmatrix}
    0&0&\cdots &0&-n N^{n-1}\\
    I&0&\cdots &0&-(n-1)N^{n-2}\\
    0&I&\cdots &0&-(n-2)N^{n-3}\\
    &&\cdots&&\\
    0&0&\cdots &I&-2N
\end{pmatrix}
\otimes I_{(n-2)!}\]
By column operations we transform it to 
\[\begin{pmatrix}
    0&0&\cdots &0&-n N^{n-1}\\
    I&0&\cdots &0&0\\
    0&I&\cdots &0&0\\
    &&\cdots&&\\
    0&0&\cdots &I&0
\end{pmatrix}
\otimes I_{(n-2)!}\]
Since rank$(N^{n-1})=1$, we see $\dim(\Ker(x_1-x_2))=(n-1)(n-2)!$. The lemma follows.
\end{proof}
\begin{lem}\label{distributiveinequality}
     $\dim (V_1\cap(V_2+\cdots +V_i))\leq (i-1)(2n-m)(n-2)!$.
\end{lem}
\begin{proof}
Define $J_i$ as in the proof of Lemma \ref{dimcount} and for $\beta \geq 2$ put $J_{[2,\beta]}=J_2 J_3\cdots J_{\beta}$. Consider the following short exact sequences:
\[ 0\to V_1\cap V_2\to V_1\cap(V_2+\cdots+V_i)\xrightarrow{J_2}(V_1\cap (V_2+\cdots +V_i))J_2\to 0,\]
\[  0\to[(V_1\cap (V_2+\cdots +V_i))J_2]\cap V_3\to (V_1\cap (V_2+\cdots +V_i))J_2\xrightarrow{J_3}(V_1\cap (V_2+\cdots +V_i))J_{[2,3]}\to 0,\]
\[\cdots\cdots\]
\[  0\to[(V_1\cap (V_2+\cdots +V_i))J_{[2,i-2]}]\cap V_{i}\to (V_1\cap (V_2+\cdots +V_i))J_{[2,i-2]}\xrightarrow{J_{i-1}}(V_1\cap (V_2+\cdots +V_i))J_{[2,i-1]}\to 0.\]
 
We claim that when $j=1,2,\dots, i-1$
\begin{align}\label{containments}
[(V_1\cap (V_2+\cdots +V_i))J_{[2,i-1]}]\cap V_{j}\subset [V_1\cap \Ker(J_{[2,j]})]J_{[2,j-1]}\subset V_1\cap V_j.
\end{align}
First of all, $\Ker(J_{[2,j]})J_{[2,j-1]}\subset \Ker(J_j)=V_j$, from which the second inclusion follows. As for the first inclusion, take ${v}_1={v}_2+\cdots +{v}_i\in V_1\cap (V_2+\cdots +V_i)$ where ${v}_j\in V_j$ for all $j$. If ${v}_1J_{[2,j-1]}\in V_j$, then \[({v}_{j+1}+\cdots +{v}_i)J_{[2,j]}=0,\] i.e. ${v}_{j+1}+\cdots +{v}_i\in \Ker(J_{[2,j]})$. As a result, \[v_1\in V_1\cap (V_2+\cdots +V_j+\Ker(J_{[2,j]}))=V_1\cap \Ker(J_{[2,j]})\] and ${v}_1J_{[2,j-1]}\in [V_1\cap \Ker(J_{[2,j]})]J_{[2,j-1]}$. Hence the claim is proved.

As a consequence of the short exact sequences and the claim,
\begin{align}
    &\dim (V_1\cap(V_2+\cdots +V_i))\nonumber\\
    =&\dim(V_1\cap V_2)+\dim ((V_1\cap(V_2+V_3))J_2)+\cdots + \dim ((V_1\cap \Ker(J_{[2,i]} )J_{[2,i-1]})\nonumber\\
  \label{bound}
  \leq  &\dim(V_1\cap V_2)+\dim(V_1\cap V_3)+\cdots +\dim(V_1\cap V_i).
    \end{align}
The last line equals $(i-1)(2n-m)(n-2)!$ by Lemma \ref{dimcount}.
\end{proof}
\begin{prop}\label{distributive}
We have a vector space decomposition:
\begin{align}
    \label{directsum}
V_1\cap (V_2+\cdots +V_i)\cong (V_1\cap V_2)\oplus \cdots \oplus (V_1\cap V_i).\end{align}
\end{prop}
\begin{proof}
We compute
\begin{align*}
&\rank(A)=\dim(V_1+\cdots +V_n)\\
=&\dim(V_1)+(\dim(V_2)-\dim(V_1\cap V_2))+\cdots +(\dim (V_n)-\dim(V_n\cap(V_1+\cdots +V_{n-1})))\\
=&n(2n-m)(n-1)!-(\dim(V_1\cap V_2)+\dim(V_3\cap(V_1+V_2))+\cdots +\dim(V_n\cap(V_1+\cdots +V_{n-1})))\\
\geq & n(2n-m)(n-1)!-\bigg((2n-m)(n-2)!+2\cdot (2n-m)(n-2)!+\cdots +(n-1)\cdot (2n-m)(n-2)!\bigg)\\
=&(2n-m)\frac{n(2n-2-(n-1))}{2}(n-2)!=\frac{2n-m}{2}n!\end{align*}
Here the second equality follows from Lemma \ref{dimcount} and the third equality follows from Lemma \ref{distributiveinequality}. 
By Corollary \ref{less}, we conclude that $\dim(V_1+\cdots +V_n)=\frac{2n-m}{2}n!$. Therefore, (\ref{bound}) is actually an equality for $i=1,\cdots, n$, the inclusions (\ref{containments}) are equalities. Thus the short exact sequences in the proof of Lemma \ref{distributiveinequality} yields (\ref{directsum}).
\end{proof}
\begin{coro}\label{equal}
We have $\dim_\C(V_1+\cdots +V_k)=\frac{2n-m}{2}k(2n-1-k)n!$ for $1\leq k\leq n$. 

  In particular $\rank(A)=\frac{2n-m}{2}n!$ and $\im(A)=\Ker(B)$.
\end{coro}
\begin{proof}We compute
    \begin{align*}
&\dim(V_1+\cdots +V_k)=\dim(V_1)+(\dim(V_2)-\dim(V_1\cap V_2))+\cdots +(\dim (V_k)-\dim(V_k\cap(V_1+\cdots +V_{k-1})))\\
=&k(2n-m)(n-1)!-(\dim(V_1\cap V_2)+\dim(V_3\cap(V_1+V_2))+\cdots +\dim(V_k\cap(V_1+\cdots +V_{k-1})))\\
=& k(2n-m)(n-1)!-\bigg((2n-m)(n-2)!+2\cdot (2n-m)(n-2)!+\cdots +(k-1)\cdot (2n-m)(n-2)!\bigg)\\
=&(2n-m)\frac{k(2n-2-(k-1))}{2}(n-2)!=\frac{2n-m}{2}k(2n-1-k)n!\qedhere
\end{align*}
\end{proof}
\begin{rem}
   The dimension count of Corollary \ref{equal} implies that  $\im(A)=\Ker(B)$  is a Lagrangian subspace in $\Hrm^*(\Bcal)^{\oplus(2n-m)}$ (where $\Bcal$ is the flag variety) with respect to the Poincare form defined by 
\[(\alpha,\beta)=d,\quad \text{if } \alpha\cup \beta =d \cdot\delta^{\oplus(2n-m)}.\]
\end{rem}

\begin{coro}\label{corollaryofAB}
   If $\sum h_kp_k \in \Hscr_c\cdot (I^W_{c-1}\cap \C[\hfrak]^W_1)$ and  $(y_i-y_{i+1})({\sum h_kp_k)}\in \afrak$ holds for all $i$, then $\ul{\sum  h_kp_k }\in I_c$. 
\end{coro}
\begin{proof}
    By Lemma \ref{ABequivalent} and Corollary \ref{equal}.
\end{proof}
\subsection{Concluding the identity on $F_1$}\begin{coro} We have  
    $F_1^{\ul{\ind}}\Lrm_c=F_1^\afrak\Lrm_c$ for all $c>1$.
\end{coro}
\begin{proof}
    Lemma \ref{containment} says $F_1^{\ul{\ind}}\Lrm_c\subset F_1^\afrak\Lrm_c$. Lemma \ref{F0} guarantees that the assumptions of Lemma \ref{notinI} and Proposition \ref{yphi} hold. Lemma \ref{notinI} implies $\ol{\Phi}_c(\Scal_{1,\perp_c})\subset F_1^{\ul{\ind}}\Lrm_c$. By Lemma \ref{Icperp}, it remains to show that $\ol{\Phi}_c(\Scal_1)\subset F_1^{\ul{\ind}}\Lrm_c$. Indeed, Corollary \ref{corollaryofAB} implies part (ii) of Proposition \ref{yphi}, from which (ii) follows as desired.
    \end{proof}
\begin{ex}\label{3,5}(See figure \ref{5/3})
	Consider $c=\frac{5}{3}$, where dim$(\Lrm)=5^{3-1}=25$. Based on Example \ref{3,4}, it remains to determine where $p_3\xi_i$ lies (recall that $\xi_i=x_i-x_{i+1}$). Note that $\Phi_c(p_2)(p_3\xi_i)\in \e_\st\Lrm_c$ is of weight $-2$ and hence can not be symmetric. Therefore $\Phi_c(p_2^2)(p_3\xi_i)$ must be $0$ and thus $p_3\xi_i\in(\afrak^2)^{\perp_c}$. Since $p_3\in I_{\frac{2}{3}}^W$, it follows from Corollary \ref{corollaryofAB} that $(y_k-y_{k+1})(p_3\xi_i)\in \Hscr_c$ for $k=1,2$. This is not obvious by direct computation.
\end{ex}
\begin{figure}
\centering
	\renewcommand{\width}{1.7}
	\newcommand{\height}{1.2}
	\begin{tikzpicture}
		\foreach \xy in {-4,...,4}{
			\node at (\width*\xy,-1*\height) {\xy};
		}
		\node at (-4*\width,1*\height) {1};
		\node at (-2*\width,1*\height) {$p_2$};
		\node at (-1*\width,1*\height) {$p_3$};
		\node at (0*\width,1*\height) {$p_2^2$};
		\node at (1*\width,1*\height) {$p_2p_3$};
		\node at (2*\width,1*\height) {$p_2^3$};
		\node at (4*\width,1*\height) {$p_2^4$};
		\node at (-3*\width,2*\height) {$\xi_i$};
		\node at (-1*\width,2*\height) {$\xi_ip_2$};
		\node at (0,2*\height) {$\xi_ip_3$};
		\node at (1*\width,2*\height) {$\xi_ip_2^2$};
		\node at (3*\width,2*\height) {$\xi_ip_2^3$};
		\node at (-2*\width,3*\height) {$\alpha_i$};
		\node at (0*\width,3*\height) {$\alpha_ip_2$};
		\node at (2*\width,3*\height) {$\alpha_ip_2^2$};
		\node at (-\width,4*\height) {$\delta$};
		\node at (\width,4*\height) {$\delta p_2$};
		\draw (-4.5*\width,2.6*\height) -- (-3.4*\width,1.3*\height) -- (-1.6*\width,1.3*\height) -- (-\width,.6*\height) -- (-.4*\width,1.3*\height) --(.4*\width,1.3*\height) -- (\width,.6*\height) -- (1.6*\width,1.3*\height) -- (3.4*\width,1.3*\height) -- (4.5*\width,2.6*\height);
		\draw (-3.5*\width,3.6*\height) -- (-2.4*\width,2.3*\height) -- (-.6*\width,2.3*\height) -- (0,1.6*\height) -- (.6*\width,2.3*\height) --(2.4*\width,2.3*\height) -- (3.5*\width,3.6*\height);
		\draw (-2.5*\width,4.6*\height) -- (-1.4*\width,3.3*\height) -- (-.6*\width,3.3*\height) -- (0,2.6*\height) -- (.6*\width,3.3*\height) --(1.4*\width,3.3*\height) -- (2.5*\width,4.6*\height);
		\draw[dashed, opacity=.5] (-.1*\width,4.5*\height) -- (4.1*\width,-.3*\height) -- (4.5*\width,-.3*\height);
		\draw[dashed, opacity=.5] (-2.6*\width,4.5*\height) -- (1.6*\width,-.4*\height) -- (4.5*\width,-.4*\height);
		\draw[dashed, opacity=.5] (-3.6*\width,3.5*\height) -- (-.2*\width,-.5*\height) -- (4.5*\width,-.5*\height);
		\draw[dashed, opacity=.5] (-4.6*\width,2.5*\height) -- (-2.1*\width,-.6*\height) -- (4.5*\width,-.6*\height);
		\node at (4.1*\width,2.6*\height) {$\color{blue}{F^{\mathrm{alg}}_3}$};
		\node at (3.1*\width,3.6*\height) {$\color{blue}{F^{\mathrm{alg}}_2}$};
		\node at (2.1*\width,4.6*\height) {$\color{blue}{F^{\mathrm{alg}}_1}$};
		\node at (4.3*\width,-.05*\height) {$\color{red}{F^{\mathfrak{a}}_0}$};
		\node at (1.9*\width,-.15*\height) {$\color{red}{F^{\mathfrak{a}}_1}$};
		\node at (0*\width,-.25*\height) {$\color{red}{F^{\mathfrak{a}}_2}$};
		\node at (-1.9*\width,-.35*\height) {$\color{red}{F^{\mathfrak{a}}_3}$};
	\end{tikzpicture}
\caption{Filtrations in example \ref{3,5}. Numbers on the bottom indicate $\hbf$-weights}
\label{5/3}
\end{figure}
\vspace{-5pt}
    \subsection{Springer fiber at the minimal nilpotent orbit}\label{minimalspringer}
There is a mysterious relationship between the matrices $A$, $B$ and the minimal Springer fiber.

Let $G=\SL_n$ with Lie algebra $\gfrak$ and $B\subset G$ the Borel consisting of upper triangular matrices. Let $\lieb$ be the Lie algebra of $B$ with nilpotent radical $\lien=[\lieb,\lieb]$. 

Let $\Nscr$ be the nilpotent cone in $\gfrak$, there is a unique minimal nilpotent $G$-orbit consisting of nilpotent elements in $\gfrak$ of rank $1$. Let $e_{min}$ be an element of this orbit and $\Bcal_{min}$ the fiber of the Springer map $G\times^B \lien\to \Nscr$ over $e_{min}$. 
By \cite{conciniprocesiconjugacy} one knows that 
\[ \Hrm^*(\Bcal_{min})=\C[\hfrak]/(\afrak+(x_1^{n-1},\cdots, x_n^{n-1})).\]
 In the case of $m=2n-1$, the right hand side is exactly $\Rcal/\im(A)=\Rcal/\Ker(B)\cong \im(B)$. Now because $\Hrm^*(\Bcal_{min})\cong \ind_{S_2}^{S_n}\mathrm{triv}$, its dimension is exactly $\frac{n!}{2}$. This gives a different proof of Corollary \ref{equal} when $m=2n-1$. We do not know how to generalize it to arbitrary $n<m<2n$.
 \section{Finishing the proof of the main theorem}\label{section_F_i}
    \subsection{$F^\afrak_\ell=F^{\ul{\ind}}_\ell$ for $\ell\geq 1$}
Recall the subspace
  \[\mathcal{S}_\ell:=\langle \ul{h\psi}\in (\afrak^{\ell+1})_\ell^{\perp_c}|h\in \Hscr_c, \psi\in I^W_{c-1}\cap \afrak^\ell \rangle \] which satisfies $\ol{\Phi}_c(\Scal_\ell)\subset F^\afrak_\ell \Lrm_c$. Our goal in this section is to show that $\ol{\Phi}_c(\Scal_\ell)\subset F_\ell^{\ul{\ind}} \Lrm_c$.  The following lemma is the easiest case of Proposition \ref{k<l}. It is stated separately to demonstrate the proof method in general.
\begin{lem}\label{k=l}
Suppose $kn<m<(k+1)n$.   If $\phi \in \mathcal{S}_{k}$ and  $(y_i-y_{i+1}) \phi\in \afrak^k$ for all $i$, then $\phi\in I_c$.
\end{lem} 
\begin{proof}
We simplify the notation $v_i:=v_i^{(c-1)}$. From Theorem \ref{arc}, we know that ${I_{c-1}^W}$ mod $\afrak^{k+1}$ is spanned by ${v_{m-n+1}},\cdots, {v_{kn}}$. Therefore 
\begin{align}
    \label{phi}
\phi\equiv \sum_{\alpha=m-n+1}^{kn} g_\alpha v_\alpha\quad \text{mod $\afrak^{k+1}$}
\end{align}
where $g_\alpha\in \Hscr_c$.

For $\alpha=m-n+1,\cdots, kn$, by definition \begin{align*}
    \label{valpha}
v_{\alpha}\equiv{c-1\choose k}\sum_{j_1+\cdots +j_k=\alpha} u_{j_1}\cdots u_{j_k} \text{ mod } \afrak^{k+1}.\end{align*}
We can order $u_{j_1}\cdots u_{j_k}$ so that $j_1\geq j_2\geq\cdots\geq j_k$. Under the lexicographic order, $u_n^{k-1}u_{\alpha-(k-1)n}$ is the maximal term. Since $p_k\equiv ku_k$ mod $\afrak$, we have that 
\[\phi= (\frac{g_{m-n+1}}{m-n+1}p_{m-kn+1}+\cdots +\frac{g_{kn}}{n}p_n)u_n^{k-1}+\text{lower terms in lexicographic order}, \quad \text{modulo $\afrak^{k+1}$ }.\]
Modulo $\afrak^{k+1}$, the part in $(y_i-y_{i+1})\phi$ that is divisible by $u_n^{k-1}$ is the product of $u_n^{k-1}$ and
\begin{align}\label{pn}
    \big(g_{m-n+1}(x_i^{m-kn}-x_{i+1}^{m-kn})+\cdots +g_{kn-1}(x_i^{n-2}-x_{i+1}^{n-2})+kg_{kn}(x_i^{n-1}-x_{i+1}^{n-1})\big).
\end{align}
The factor $k$ in the last summand comes from the multiplicity of $p_n$ in $p_n^k$ being $k$.

Hence that $(y_i-y_{i+1})\phi\in \afrak^k$ implies that for all $i$,  the polynomial (\ref{pn}) lies in $\afrak$. This means that
\[(g_{m-n+1},\cdots, g_{kn-1},kg_{kn})\in \Ker(B_{m-kn,n-1}) \quad \text{modulo $\afrak$}.\]

By Corollary \ref{equal}, there exist $h_j\in\Hscr_c$ such that one has an equality
\[(g_{m-n+1},\cdots, g_{kn-1},kg_{kn})\equiv(h_1,\cdots,h_n)A_{m-kn,n-1} \quad \text{modulo $\afrak$}.\]
Therefore, by (\ref{generalfi}) and (\ref{phi})  we have 
\begin{align*}
    \phi\equiv&\sum_j \big(h_jx_j^{n-1}v_{m-n+1}+\cdots +h_jx_j^{m-n}v_{kn}\big)-(k-1)v_{kn} \quad \text{modulo $\afrak^{k+1}$}.\\
    \equiv&\sum_j h_j(x_j^{n-1}v_{m-n+1}+\cdots +x_j^{m-n}v_{kn})-(k-1)v_{kn} \quad \text{modulo $\afrak^{k+1}$}.
    \end{align*}
    Combining with equation (\ref{fmodulo}) $f_i\equiv \sum_{j =m-nk}^{n-1} x_i^j v_{m-j}\quad \text{mod }\afrak^{k +1}$, we see that \[\phi\equiv\sum_j h_jf_j-(k-1)g_{kn}v_{kn} \quad \text{modulo $\afrak^{k+1}$}.\]
It remains to show $\ul{g_{kn}v_{kn}}\in I_c$. Notice that \[(k-1)(y_i-y_{i+1})({g_{kn}v_{kn}})=(y_i-y_{i+1})(\phi-\sum h_jf_j)\in \afrak^k\] for all $i$ implies that $(0,\cdots, 0,g_{kn})\in \Ker (B_{m-kn,n-1})$. 
Applying Proposition \ref{equal}, we deduce that there exists $h'_i\in \Hscr_c$ such that \[(0,\cdots,0,g_{kn})=(h'_1,\cdots, h'_n)A_{m-kn,n-1}.\] Therefore ${g_{kn}v_{kn}}\equiv \sum h'_i f_i$ mod $\afrak^k$ and $\ul{g_{kn}v_{kn}}\in I_c$.
\end{proof}
\begin{prop}\label{k<l}
    Suppose $kn<m<(k+1)n$, $\ell\geq  k$ and 
    $\phi\in \mathcal{S}_{\ell}$.
 If  $(y_j-y_{j+1}) \phi\in \afrak^\ell$ for all $j$, then ${\phi }\in I_c$. In other words, $\ol{\Phi}_c((y_k-y_{k+1})\phi)\in F^\afrak_{\ell-1}\Lrm_c$.
\end{prop}
\begin{proof}
 By Theorem \ref{arc}, ${I_{c-1}^W}$ is spanned by ${v_{m-n+1}},\cdots, {v_{m-1}}$. 
 
 Suppose \[\phi\equiv \sum_{i=m-n+1}^{m-1} \phi_iv_i\text{ mod }\afrak^{\ell+1}\] where \[\phi_{m-n+1},\cdots, \phi_{nk}\in \afrak^{\ell-k},\quad \phi_{nk+1},\cdots, \phi_{m-1}\in \afrak^{\ell-k+1}.\] Assume $(I_1,\cdots,I_\ell)$ ($I_1\geq I_2\cdots \geq I_\ell$) is the largest sequence under the lexicographic order such that $u_{I_1}\cdots u_{I_\ell}$ shows up in $\phi$ with a nonzero coefficient in $\Hscr_c$. 
 
 \textbf{We will use $f_i$ to annihilate this nonzero term in $\phi$ without adding higher order terms so that the lemma will then follow from induction on $(I_1,\cdots,I_\ell)$. }

By definition, one has
\begin{align*}
    \label{valpha}
v_\alpha\equiv{c-1\choose s}\sum_{j_1+\cdots +j_{s}=\alpha} u_{j_1}\cdots u_{j_{s}} \text{ mod } \afrak^{s+1}\end{align*}
where $s=k$ when $m-n+1\leq \alpha\leq nk$ and $s=k+1$ when $nk+1\leq \alpha\leq m-1$. 
Suppose the nonzero coefficient of $u_{I_1}\cdots u_{I_\ell}$ comes from $\phi_\alpha v_\alpha$. Then there are three possibilities:
\\Case $1$:  $(u_{\alpha-nk}u_n^k)|(u_{I_1}\cdots u_{I_\ell})$ when $nk+2\leq \alpha\leq m-1$; \\
Case $2$: $(u_2u_{n-1}u_n^{k-1})|(u_{I_1}\cdots u_{I_\ell})$ when $\alpha=nk+1$ \\
Case $3$: $(u_{\alpha-n(k-1)}u_n^{k-1})|(u_{I_1}\cdots u_{I_\ell})$ when $m-n+1\leq \alpha\leq nk$. 

\vspace{10pt}
\textbf{ Case $1$}: We may assume $u_{I_1}\cdots u_{I_\ell}=u_{I_1}\cdots u_{I_{\ell-k-1}}u_n^ku_{\alpha-nk}$. Up to some linear rearranging, we may also assume the nonzero coefficient of this term completely comes from $\phi_\alpha v_\alpha$. In other words, the coefficients of $u_{I_1}\cdots u_{I_\ell}$ in $\phi_\gamma v_\gamma$ for $\gamma\neq \alpha$ add up to be $0$. Similarly, if the coefficient of $u_{I_1}\cdots u_{I_{\ell-k-1}}u_n^ku_\beta$ in $\phi$ is nonzero for $\beta< \alpha-nk$, we may assume the coefficient is completely contributed by  $u_{I_1}\cdots u_{I_{\ell-k-1}}v_t$ for $nk+2\leq t\leq m-1$ or $u_{I_1}\cdots u_{I_{\ell-k-1}}u_nv_t$ for $m-n+1\leq t\leq nk$ (depending on the size of $\beta$).

Inside $(y_j-y_{j+1})\phi$, $\big( (y_j-y_{j+1}) u_{\alpha-nk}\big) u_{I_1}\cdots u_{I_{\ell-k-1}}u_n^k$ can only be combined with like terms in the form of $\big( (y_j-y_{j+1}) u_\beta\big) u_{I_1}\cdots u_{I_{\ell-k-1}}u_n^k$ for $\beta< \alpha-nk$. 

Assume the coefficients of $u_{I_1}\cdots u_{I_{\ell-k-1}}u_nv_t$ for $m-n+1\leq t\leq nk$ and $u_{I_1}\cdots u_{I_{\ell-k-1}}v_t$ for $nk+2\leq t\leq m-1$ are $h_{m-n+1},\cdots h_{nk}, h_{nk+2},\cdots, h_{m-1}$(with some of them being possibly $0$.). To have $(y_j-y_{j+1})\phi\in\afrak^\ell$, by Lemma \ref{ABequivalent}, there has to be \[(z_{nk+2}h_{nk+2},\dots, z_{m-1}h_{m-1},0,z_{m-n+1}h_{m-n+1},\cdots, z_{nk}h_{nk})\in \Ker(B_{1,n-1})\quad \text{mod $\afrak$}.\] for some positive integers $z_{m-n+1},\cdots,z_{nk},z_{nk+2},\cdots, z_{m-1}$ due to possible nontrivial multiplicities (as in (\ref{pn})).

Since $\Ker(B_{1,n-1})=\im(A_{1,n-1})$, by Corollary \ref{equal} there exists $(g_1,\cdots, g_n)$ such that \[(z_{nk+2}h_{nk+2},\dots, z_{m-1}h_{m-1},0,z_{m-n+1}h_{m-n+1},\cdots, z_{nk}h_{nk})\equiv (g_1,\cdots, g_n)A_{1,n-1}\quad \text{mod $\afrak$}.\]
Recall that by Lemma \ref{lastlemma}, $u_n\equiv x_i^n$ mod $\afrak$ while (\ref{generalfi}) says $f_i= \sum_{j =0}^{n-1}  ( x_i^j+x_i^{j-2} u_2+\cdots + u_j )v^{(c-1)}_{m-j}$. Therefore
 modulo $\afrak^{k+2}$ we have
 \begin{align*}
    &z_{nk+1}h_{nk+2}v_{nk+2}+\cdots +z_{m-1}h_{m-1}v_{m-1}+z_{m-n+1}h_{m-n+1}u_nh_{m-n+1}+\cdots +z_{nk}h_{nk}u_nu_{nk}\\
    \equiv &\sum_i g_i (x_i^{n-1}v_{nk+2}+\cdots +x_i^{(k+1)n-m+2}v_{m-1}+x_i^{n(k+1)-m}u_nv_{m-n+1}+\cdots +x_iu_nv_{nk})\\
    \equiv&\sum_i g_i x_i^{n(k+1)-m+1}f_i.
    \end{align*}
Now write $(*):=\phi-\frac{1}{z_\alpha}u_{I_1}\cdots u_{I_{\ell-k+1}}\sum g_i x_i^{n(k+1)-m+1}f_i$. Then $(*)$ still satisfies that $(y_j-y_{j+1})(*)\in \afrak^\ell$ for all $j$ and the coefficient of $u_{I_1}\cdots u_{I_{\ell-k-1}}u_n^ku_{\alpha-nk}$ in $(*)$ is $0$ with no higher-order summand added compared to $\phi$. Our goal is achieved.

\vspace{10pt}
\textbf{Case $2$}: We may assume $u_{I_1}\cdots u_{I_\ell}=u_{I_1}\cdots u_{I_{\ell-k-1}}u_2u_{n-1}u_n^{k-1}$. Similar as case $1$, inside $(y_j-y_{j+1})\phi$, the terms $\big( (y_j-y_{j+1}) u_{2}\big) u_{I_1}\cdots u_{I_{\ell-k-1}}u_n^k$ can only be combined with like terms in the form of $\big( (y_j-y_{j+1}) u_\beta\big) u_{I_1}\cdots u_{I_{\ell-k-1}}u_{n-1}u_n^{k-1}$ for $\beta<2$, which are $0$. Set the coefficient of $u_{I_1}\cdots u_{I_{\ell-k-1}}u_2u_{n-1}u_n^{k-1}$ to be $h_{nk+1}$. Then 
\[(0,\cdots, 0, h_{nk+1},0,\cdots,0)\in \Ker(B_{1,n-1})=\im(A_{1,n-1}).\]
Here $h_{nk+1}$ lies in the $(n(k+1)-m+1)$-th entry. 
Therefore there exists $g_i\in \Hscr_c$ such that 
\[h_{nk+1}v_{nk+1}\equiv \sum g_i(x_i^{n-1}v_{m-n+1}+\cdots +x_i^{m-nk-1}v_{nk+1}+\cdots x_i v_{m-1})=\sum g_i f_i \quad \text{mod $\afrak^{k+2}$}.\]
Now the highest degree in $\phi- u_{I_1}\cdots u_{I_{\ell-k-1}}\sum g_i f_i$ is lower than the original $(I_1,\cdots, I_\ell)$.

\vspace{10pt}
\textbf{Case $3$}: We may assume that $u_{I_1}\cdots u_{I_\ell}=u_{I_1}\cdots u_{I_{\ell-k}} u_n^{k-1}u_{\alpha-n(k-1)}$. Given case $1$, we may further assume that $I_1,\cdots, I_\ell<n$. Moreover, if the coefficient of $u_{I_1}\cdots u_{I_{\ell-k}}u_n^{k-1}u_\beta$ in $\phi$ is nonzero, we may assume that the coefficient is completely contributed by  $u_{I_1}\cdots u_{I_{\ell-k}}v_t$ for $m-n+1\leq t\leq nk$.

Inside $(y_j-y_{j+1})\phi$, $\big( (y_j-y_{j+1}) u_{\alpha-nk}\big) u_{I_1}\cdots u_{I_{\ell-k}}u_n^{k-1}$ can only be combined with like terms in the form of $\big( (y_j-y_{j+1}) u_\beta\big) u_{I_1}\cdots u_{I_{\ell-k}}u_n^{k-1}$ for $\beta< \alpha-nk$. Assume the coefficient of $u_{I_1}\cdots u_{I_{\ell-k}}v_t$ is $h_t$, $t=m-n+1,\cdots, nk$ (with some of the $h_t$ possibly being $0$). Then \[(z_{m-n+1}h_{m-n+1},\cdots, z_{nk}h_{nk})\in \Ker(B_{m-nk,n-1})=\im(A_{m-nk,n-1}).\] for some positive rational numbers $z_{m-n+1},\cdots,z_{nk}$ due to possible nontrivial multiplicities. Therefore
  there exists $g_i\in \Hscr_c$ such that 
 \begin{align*}
    z_{m-n+1}h_{m-n+1}h_{m-n+1}+\cdots +z_{nk}h_{nk}u_{nk}
    \equiv \sum_i g_i (x_i^{n-1}v_{m-n+1}+\cdots +x_i^{m-n(k-1)+1}v_{nk})\quad \text{mod $\afrak^{k+1}$},
    \end{align*}
    and \[u_{I_1}\cdots u_{I_{\ell-k}}(z_{m-n+1}h_{m-n+1}u_{m-n+1}+\cdots +z_{nk}h_{nk}u_{nk})\equiv u_{I_1}\cdots u_{I_{\ell-k}}\sum_i g_if_i\quad \text{mod $\afrak^{\ell+1}$}.\]
    Similar to the previous cases the highest degree in $\phi- \frac{1}{z_\alpha}u_{I_1}\cdots u_{I_{\ell-k}}\sum_i g_i f_i$ is lower than the original $(I_1,\cdots, I_\ell)$.
   \end{proof}
\subsection{Proof of the main theorem}
With all the puzzle pieces collected, we conclude the proof of Theorem \ref{thm}.
\begin{proof}
    By Lemma \ref{containment}, we need to show that  $F_\ell^\afrak\Lrm_c\subset F_\ell^{\ul{\ind}}\Lrm_c$ for all $\ell\geq 0$ when $c>1$ and $F_\ell^\afrak\e\Lrm_c\subset F_\ell^{\ul{\ind}}\e\Lrm_c$ for all $\ell\geq 0$ when $c>0$. We do so by induction on both $F_\ell$ and $c$, using the order defined by (\ref{order}).
    
    The base case of $F_0$ is shown in Lemma \ref{F0} for any $c>1$. On the other hand, for all $c>0$\\ $F_0^\afrak\e\Lrm_c=\C\cdot 1=F_0^{\ul{\ind}} \e\Lrm_c$. 

     When $c=\frac{m}{n}<1$ with $m>1$, since both $F^\afrak$ and $F^{\ul{\ind}}$ are compatible with the isomorphism $\e_n\Lrm_{\frac{m}{n}}\cong \e_m\Lrm_{\frac{n}{m}}$, the inclusion $F_\ell^\afrak\e\Lrm_c\subset F_\ell^{\ul{\ind}}\e\Lrm_c$ can be restated in terms of $\e_m\Lrm_\frac{n}{m}$. Thus, we may assume $c>1$.
    
    For $c>1$ and $\ell>0$, assume we have shown that $F^\afrak_i\Lrm_{c'}=F^{\ul{\ind}}_i\Lrm_{c'}$  for all $i\geq 0$ and $c'$ such that $c'\prec c$, and assume also that   $F_{\ell'}^\afrak\Lrm_{c}=F_{\ell'}^{\ul{\ind}}\Lrm_{c}$  for $0\leq \ell'<\ell $. 

    Take a homogeneous $\phi\in (\afrak^{\ell+1})^{\perp_c}$ such that $0\neq \ol{\Phi}_c(\phi)\in F_\ell^\afrak\Lrm_c$.  By Lemma \ref{notinI} and the induction hypothesis on $c$, we have $\ol{\Phi}_c(\mathcal{S}_{\ell}^{\perp_c})\in F_\ell^{\ul{\ind}}\Lrm_c$. 
    
    Hence, by Lemma \ref{Icperp}, we may assume that $\phi\in \mathcal{S}_{\ell}$, in which case $\ell\geq \lfloor c\rfloor $. 
    Proposition \ref{k<l} implies that $(y_i-y_{i+1})\phi\in (\afrak^\ell)^{\perp_c}$. 
    Proposition \ref{yphi} ((ii)$\Rightarrow$ (i)) and the inductive hypothesis on $F_\ell$ implies that $\ol{\Phi}_c(\Scal_\ell)\subset F_\ell^{\ul{\ind}}\Lrm_c$.
    This concludes the proof of Theorem \ref{thm}.
\end{proof}
\subsection{Description of $F^\afrak$ for $c<1$}Let $c=\frac{m}{n}<1$ for positive integer $m$ coprime to $n$.  Inside $\Lrm_c$, $\delta=0$ but one can define $\Hscr_c:=\afrak^{\perp_c}$. It still holds that $\Lrm_c=\Hscr_c\cdot \e\Lrm_c$. In this setting, $\Hscr_c$ is only isomorphic a proper submodule of the regular representation of $S_n$. 
\begin{prop}
For all $1\leq k\leq n-1$, $\ell>0$ and $\ol{\Phi}_c(\phi)\in F_\ell^\afrak \Lrm_c$, we have $\ol{\Phi}_c((y_k-y_{k+1})\phi)\in F^\afrak_{\ell-1}\Lrm_c$.
\end{prop}
\begin{proof}
We first note that when $c<1$ the equivalence of statements in Proposition \ref{yphi} holds for the whole $\Lrm_c$ rather than merely $\Scal_\ell$. Therefore we only need to show that if $\phi\in (\afrak^{\ell+1})^{\perp_c}\cap \afrak^\ell$ satisfies $(y_k-y_{k+1})\phi\in\afrak^\ell$ for all $k=1,\dots,n-1$ then $\phi\in I_c$.

We first prove the case when $\ell=1$. Recall the ideal $I_c\subset \C[\hfrak]$ is generated by 
\[  f_i=\sum_{j=0}^m x_i^jv^{(c)}_{m-j}, \quad i=1,\cdots, n\]
where
\[v^{(c)}_j\equiv cu_j \quad \text{mod } \afrak,\quad j=2,\cdots, n\]
Take $\phi:=\ul{\sum_{i=2}^n h_ip_i}\in (\afrak^2)^{\perp_c}$ where $h_i\in \Hscr_c$.  
By Lemma \ref{ABequivalent}, $(y_k-y_{k+1})\phi\in \afrak$ for all $k=1,\cdots, n-1$ if and only if $(2h_2,\cdots, nh_n)\in \Ker(B_{2,n-1})$. By Corollary \ref{equal}, there exists $g_1,\cdots, g_n$ such that $(2h_2,\cdots, nh_n)=(g_1,\cdots, g_n)A_{2,n-1}$. Therefore,
\begin{align*}
    \sum_{i=2}^n h_ip_i\equiv \sum_{i=2}^n ih_iu_i\equiv \sum_{j=1}^n\sum_{i=2}^n g_jx_j^{n+1-i}u_i\quad \text{mod }\afrak.
    \end{align*}
Note that $c\sum_{i=2}^m x_j^{n+1-i}u_i\equiv x_j^{n+1-m}(f_j-x_j^m)$ mod $\afrak$. Thus we obtain that 
\begin{align*}
\sum_{j=1}^n\sum_{i=2}^n g_jx_j^{n+1-i}u_i
\equiv& \frac{1}{c}\sum_{j=1}^n g_j\bigg(x_j^{n+1-m}(f_j-x_j^m)+x_j^{n-m}v_{m+1}^{(c)}+\cdots +x_jv_n^{(c)}\bigg)\quad \text{mod }\afrak\\
\equiv &\frac{1}{c}\sum_{j=1}^n g_j\bigg(x_j^{n+1-m}f_j+x_j^{n-m}v_{m+1}^{(c)}+\cdots +x_jv_n^{(c)}\bigg) 
\quad \text{mod }\afrak.
\end{align*}
The second equality follows as $x_j^{n+1}\in\afrak$.
Since the ideal $I_c^W\subset \C[\hfrak]^W$ is generated by $v_{m+1}^{(c)},\dots, v_{m+n-1}^{(c)}$, we conclude that $\ul{\sum_{i=2}^n h_ip_i}\in I_c$. 

When $\ell>1$, suppose $(y_k-y_{k+1})\phi\in \afrak^{\ell+1}$ for all $k$. Assume $(I_1,\cdots,I_\ell)$ ($I_1\geq I_2\cdots \geq I_\ell$) is the largest sequence under the lexicographic order such that $u_{I_1}\cdots u_{I_\ell}$ shows up in $\phi$ with a nonzero coefficient in $\Hscr_c$. Similar to the proof of Proposition \ref{k<l}, we will use $f_i$ to kill this nonzero term in $\phi$ without adding higher-order terms, and, as a result, the proposition will follow from induction on $(I_1,\cdots,I_\ell)$. 

For $\gamma=2,\dots, i_\ell$, let $h_\gamma\in\Hscr_c$ be the coefficient of $u_{I_1}\cdots u_{I_{\ell-1}}u_\gamma$ in $\phi$. Inside $(y_k-y_{k+1})\phi$, $h_{i_\ell}\big( (y_k-y_{k+1}) u_{I_\ell}\big) u_{I_1}\cdots u_{I_{\ell-1}}$ can only be combined with like terms in the form of $h_\gamma\big( (y_k-y_{k+1}) u_\gamma\big) u_{I_1}\cdots u_{I_{\ell-1}}$ for $\gamma=2,\dots, i_\ell-1$. Therefore $(y_k-y_{k+1})\phi\in \afrak^{\ell+1}$ for all $k$ implies that 
\[(z_2h_2,\cdots, z_nh_n)\in \Ker(B_{1,n-1})\quad \text{mod }\afrak\]
where $z_i$'s are integers with $z_{i_\ell}\neq 0$ and $z_\gamma=0$ for $\gamma>i_\ell$. By the discussion above in the case of $\ell=1$, we can express $\ul{\sum_{\gamma=2}^n z_\gamma h_\gamma u_\gamma}$ by $f_1,\dots, f_n$. Write $(*):=\phi-\frac{1}{z_{i_\ell}}(\sum_{\gamma=2}^n z_\gamma h_\gamma u_\gamma)u_{I_1}\cdots u_{I_{\ell-1}}$. Then $(*)$ still satisfies $(y_k-y_{k+1})(*)\in\afrak^\ell$ for all $k$ and the coefficients of $u_{I_1}\cdots u_{I_\ell}$ in $(*)$ is $0$ with no higher-order term added compared to $\phi$. This concludes the proof of the proposition.
\end{proof}

\appendix
\section{Distributive lattices}\label{appendix}
    We use the language of distributive lattice to further describe the structure captured by (\ref{directsum}) in Proposition \ref{distributive}. It would be more satisfying to find a conceptual explanation for the occurrence of this phenomenon.
    
    Following \cite{bgs1} \cite{bezrulattice}, we define the following notions:
    \begin{itemize}
        \item A set of subspaces of vector space $U$ is called a lattice if it is closed under taking sum and intersections.
        \item A lattice is called distributive if $U_1\cap(U_2+U_3)=(U_1\cap U_2)+(U_1\cap U_3)$ for any elements $U_1,U_2,U_3$ in the lattice.
        \item A $n$-tuple $(U_1,\cdots, U_n)$ of subspaces in $U$ is called distributive if the lattice generated by $U_1,\cdots, U_n$ is distributive.
        \item A $n$-tuple is called predistributive if all the $(n-1)$-tuples $(U_1,\dots, \hat{U_i},\cdots, U_n)$, $i=1,\cdots, n$, are distributive.
        \item A $n$-tuple is called acyclic if all the $3$-triples $(U_1\cap \cdots \cap U_i, U_{i+1}, U_{i+2}+\cdots +U_n)$, $i=1,\cdots, n-2$, are distributive.
      \end{itemize}
      \begin{lem}\cite[Lemma 4.5.2]{bgs1}\label{predist}
    A $n$-tuple $(U_1,\cdots, U_n)$ of subspaces in $U$ is distributive if and only if it is both predistributive and acyclic.
\end{lem}

Write $V_i=\{(x_i^{n-1}\phi,\cdots x_i^{m-n}\phi)|\phi\in \Rcal\}\subset \Rcal_n^{2n-m}$ for $i=1,2,\cdots, n$. 
We will finally show that 
\begin{prop}\label{distributivelattice}
    The $n$-tuple $(V_1,\cdots, V_n)$ is distributive.
\end{prop}
To prove this, we first establish the following lemma:
\begin{lem}\label{acyclic}
For any $1\leq j<i$, we have
\[V_1\cap\cdots\cap V_j\cap (V_{j+1}+\cdots +V_i)= (V_1\cap \cdots \cap V_j\cap V_{j+1})+(V_1\cap \cdots \cap V_j\cap V_{j+2})+\cdots +(V_1\cap \cdots \cap V_j\cap V_i)\]
\end{lem}
\begin{proof}
    Consider the short exact sequences 
\[ V_1\cap V_2\cap\cdots \cap V_j\cap V_{j+1}\to V_1\cap V_2\cap\cdots\cap V_j\cap (V_{j+1}+\cdots +V_i)\xrightarrow{J_{j+1}}(V_1\cap V_2\cap\cdots\cap V_j\cap (V_{j+1}+\cdots +V_i))J_{j+1}\]
\begin{align*}
[(V_1\cap\cdots\cap V_j\cap (V_{j+1}+\cdots +V_i))J_{j+1}]\cap V_{j+2}\to &(V_1\cap\cdots\cap V_j\cap (V_{j+1}+\cdots +V_i))J_{j+1}\\
&\xrightarrow{J_{j+2}}(V_1\cap \cdots\cap V_j\cap (V_{j+1}+\cdots +V_i))J_{j+1}J_{j+2}
\end{align*} 
\[\cdots\cdots\]
\[ [(V_1\cap\cdots\cap V_j\cap (V_{j+1}+\cdots +V_i))J_{j+1}\cdots J_{i-2}]\cap V_{i}\to (V_1\cap \cdots\cap V_j\cap (V_{j+1}+\cdots +V_i))J_{j+1}\cdots J_{i-2}\]
\[\xrightarrow{J_{i-1}}(V_1\cap \cdots\cap V_j\cap (V_{j+1}+\cdots +V_i))J_{j+1}\cdots J_{i-1},\]
from which we obtain
\begin{align*}
   \dim V_1\cap\cdots\cap V_j\cap (V_{j+1}+\cdots +V_i)=& \dim V_1\cap V_2\cap\cdots \cap V_j\cap V_{j+1} \\
   &+\dim [(V_1\cap\cdots\cap V_j\cap (V_{j+1}+\cdots +V_i))J_{j+1}]\cap V_{j+2}+\cdots\\
   &+\dim (V_1\cap \cdots\cap V_j\cap (V_{j+1}+\cdots +V_i))J_{j+1}\cdots J_{i-1}.
\end{align*}
On the other hand, for $1\leq j<1$ and $1\leq k\leq i-j$, 
there are inclusions
    \begin{align*}
    &[(V_1\cap\cdots\cap V_j\cap (V_{j+1}+\cdots +V_i))J_{j+1}\cdots J_{j+k-1}]\cap V_{j+k}\\
    \subset &[V_1\cap\cdots\cap V_j \cap \Ker(J_{j+1}\cdots J_{j+k})]J_{j+1}\cdots J_{j+k-1}
    \subset V_1\cap\cdots\cap V_j\cap V_{j+k}.
    \end{align*}
As a consequence
\[\dim V_1\cap\cdots\cap V_j\cap (V_{j+1}+\cdots +V_i)\leq \sum_{k=1}^{i-j}\dim V_1\cap\cdots\cap V_j\cap V_{j+k}.\]
Therefore the inclusion
\[V_1\cap\cdots\cap V_j\cap (V_{j+1}+\cdots +V_i)\supset (V_1\cap \cdots \cap V_j\cap V_{j+1})+(V_1\cap \cdots \cap V_j\cap V_{j+2})+\cdots +(V_1\cap \cdots \cap V_j\cap V_i)\]
has to be an equality.
\end{proof}
\begin{proof}[Proof of Proposition \ref{distributivelattice}]
We show the lemma by induction. The cases of $n=1$ or $2$ is trivial. When $n=3$, a $3$-tuple $(U_1,U_2,U_3)$ is distributive if and only if $U_1\cap(U_2+U_3)=(U_1\cap U_2)+(U_1\cap U_3)$ (\cite[example after Lemma 4.5.1]{bgs1}). Our $(V_1,V_2,V_3)$ satisfies this property.

Now assume any $(n-1)$-subtuple of $(V_1,\cdots, V_n)$ is distributive. In other words, $(V_1,\cdots, V_n)$ is predistributive. By Lemma \ref{predist}, it only remains to show that $(V_1,\cdots, V_n)$ is acyclic, i.e. all the $3$-triples $(V_1\cap \cdots \cap V_i, V_{i+1}, V_{i+2}+\cdots +V_n)$, $i=1,\cdots, n-2$, are distributive. But this follows from the identity
\begin{align*}
&V_1\cap \cdots \cap V_i\cap (V_{i+1}+ (V_{i+2}+\cdots +V_n))\\
=&(V_1\cap \cdots \cap V_i\cap V_{i+1})+(V_1\cap \cdots \cap V_i\cap V_{i+2})+\cdots +(V_1\cap \cdots \cap V_i\cap V_n)\\
=&(V_1\cap \cdots \cap V_i\cap V_{i+1})+ (V_1\cap \cdots \cap V_i\cap (V_{i+2}+\cdots +V_n)))
\end{align*}
using Lemma \ref{acyclic} twice.
\end{proof}
\bibliographystyle{amsalpha}
\bibliography{refs.bib}
\address
\end{document}